\documentclass[11pt,letterpaper]{amsart}

\usepackage{graphicx}
\usepackage{amssymb,amscd,amsthm,amsxtra}
\usepackage{latexsym}
\usepackage{epsfig}
\usepackage{amsmath}
\usepackage{bm}
\usepackage{amssymb}
\usepackage{amsthm}
\usepackage[usenames,dvipsnames,svgnames,table]{xcolor}
\usepackage[linktocpage=true,colorlinks=true,linkcolor=Blue,citecolor=BrickRed,urlcolor=RoyalBlue]{hyperref}
\usepackage[alphabetic]{amsrefs}
\usepackage{comment}

\usepackage[colorinlistoftodos,prependcaption,textsize=tiny]{todonotes}
\usepackage{caption}
\usepackage{tikz}
\usetikzlibrary{positioning}
\usetikzlibrary{intersections}
\usetikzlibrary{calc,through,backgrounds}
\usepackage{enumitem} 
\usepackage[mathscr]{euscript} 

\usepackage{mathrsfs} 

\usepackage{pgfplots}
\pgfplotsset{compat=1.18}

\usepgfplotslibrary{fillbetween}
\usepackage{xcolor}  
\pgfplotsset{compat=1.18}

\usepackage{mathtools}

\numberwithin{equation}{section} 

\newtheorem{theorem}{Theorem}
\newtheorem{example}{Example}
\newtheorem{definition}{Definition}
\newtheorem{lemma}{Lemma}
\newtheorem{corollary}{Corollary}

\newtheorem{remark}{Remark}[section]

\title[ Regularity  for  Hardy-H\'{e}non-type models ruled by $\infty$-Laplacian]{Improved regularity estimates for  Hardy-H\'{e}non-type equations driven by $\infty$-Laplacian}

\author{ Elzon C. Bezerra J\'{u}nior}
\address{Universidade Federal do Cariri - UFCA. Departamento  de Matemática. Rua Oleg\'{a}rio Em\'{i}dio de Ara\'{u}jo, S/N – Centro, Brejo Santo – CE, CEP 63260-000, Brazil}
\email{cezar.bezerra@ufca.edu.br}

\author{Jo\~{a}o Vitor da Silva}
\address{Departamento de Matem\'{a}tica - Instituto de Matem\'{a}tica, Estat\'{i}stica e Computa\c{c}\~{a}o Cient\'{i}fica - Universidade Estadual de Campinas, Rua S\'{e}rgio Buarque de Holanda, 651,  CEP 13083-859, Campinas, SP, Brazil}
\email{jdasilva@unicmp.br}

\author{ Thialita M. Nascimento}
\address{Department of Mathematics,  Iowa State University, 96 Carver Hall, 50011, Ames, IA, USA}
\email{thnasc@iastate.edu}

\author{Ginaldo S. S\'{a}}
\address{Departamento de Matem\'{a}tica - Instituto de Matem\'{a}tica, Estat\'{i}stica e Computa\c{c}\~{a}o Cient\'{i}fica - Universidade Estadual de Campinas, Rua S\'{e}rgio Buarque de Holanda, 651, CEP 13083-859, Campinas, SP, Brazil}
\email{ginaldo@unicamp.br}

\begin{document}
\maketitle

\date{} 

\begin{abstract}  	
{\scriptsize{
In this work, we establish sharp and improved regularity estimates for viscosity solutions of Hardy-H\'{e}non-type equations with possibly singular weights and strong absorption governed by the $\infty$-Laplacian
$$
		\Delta_{\infty} u(x)  = |x|^{\alpha}u_+^m(x) \quad \text{in} \quad B_1,
$$
under suitable assumptions on the data.  In this setting, we derive an explicit regularity exponent that depends only on universal parameters. Additionally, we prove non-degeneracy properties, providing further geometric insights into the nature of these solutions. Our regularity estimates not only improve but also extend, to some extent, the previously obtained results for zero-obstacle and dead-core problems driven by the $\infty$-Laplacian. As an application of our findings, we also address some Liouville-type results for this class of equations.
}}


\tableofcontents

\end{abstract}

\newpage

\section{Introduction}

In this work, we establish sharp and improved regularity estimates for viscosity solutions of Hardy-H\'{e}non-type elliptic equations, governed by the infinity-Laplacian under a strong absorption condition:
\begin{equation}\label{pobst}
	 \Delta_{\infty}  u(x)  =  f(|x|, u(x)) \quad \text{in} \quad B_1,
\end{equation}
where $B_1 \subset \mathbb{R}^n$ denotes the unit $n$-dimensional ball centered at the origin with $n \geq 2$, and
\[
\Delta_{\infty} u(x) := \sum_{i,j = 1}^n \partial_i u(x) \, \partial_{ij} u(x) \, \partial_j u(x) = (Du(x))^T D^2 u(x) \cdot D u(x),
\]
is the infinity-Laplacian operator (see \cite{ACJ04} for an enlightening survey on this topic). Moreover, for all $(x, t) \in B_1 \times \mathfrak{I}$, $r, s \in (0, 1)$ ($\mathfrak{I} \subset \mathbb{R}$ an interval), we assume that there exist a universal constant $\mathrm{c}_n > 0$ and an $f_0 \in L^\infty(B_1)$ such that
{\scriptsize{
\begin{equation}\label{EqHomog-f}
|f(r|x|, s t)| \leq \mathrm{c}_n r^{\alpha} s^m \|f_0\|_{L^\infty(B_1)} \quad \text{for} \quad 0 \leq m < 3 \quad  \text{and} \quad  \alpha \in \left(-\frac{4}{3}m, \infty\right).
\end{equation}}}

Observe that the restriction
$$
f(|x|, 1) := \mathfrak{h}(x)
$$
acts as a type of weight in our diffusion model and satisfies the following growth condition: there exist universal constants $0 < \mathrm{c}_1 \leq \mathrm{c}_2 < \infty$ such that
$$
\mathrm{c}_1|x|^{\alpha} \leq \mathfrak{h}(x) \leq \mathrm{c}_2|x|^{\alpha}.
$$

It should be noted that the exponent of the weight can be negative. This allows us to include a broader class of Hardy-H\'{e}non-type models, provided that the right-hand side remains bounded after an appropriate normalization and scaling process. The following examples will illustrate these statements and properties.

\begin{example}[{\bf A toy model}]
An archetypal model for \eqref{pobst} is the Hardy-H\'{e}non-type problem (with strong absorption) driven by the $\infty$-Laplacian:
$$
\Delta_{\infty} u(x) = \sum_{i=1}^{k_0} c_i |x|^{\alpha_i} u_{+}^{m_i}(x) \quad \text{in} \quad B_1,
$$
where
{\scriptsize{
$$
0 < \alpha_i < \infty, \quad c_i \geq 0, \quad \text{and} \quad 0 \leq m_i < 3 \quad (\text{or} \quad -\frac{4}{3} m < \alpha_i < 0) \quad \text{for} \quad 1 \leq i \leq k_0.
$$}}
In this case, $f(x, u)$ satisfies \eqref{EqHomog-f} as follows:
$$
f(r|x|, s u) \leq r^{\displaystyle \min_{1 \leq i \leq k_0} \{\alpha_i\}} s^{\displaystyle \min_{1 \leq i \leq k_0} \{m_i\}} \|f_0\|_{L^\infty(B_1)} \quad \forall (x, u) \in B_1 \times \mathfrak{I},
$$
$$
\left( \text{resp.} \,\, f(r|x|, s u) \leq r^{\displaystyle \max_{1 \leq i \leq k_0} \{\alpha_i\}} s^{\displaystyle \min_{1 \leq i \leq k_0} \{m_i\}} \|f_0\|_{L^\infty(B_1)} \right),
$$
where
$$
\|f_0\|_{L^\infty(B_1)} := \sup_{\overline{B_1}} f(|x|, u(x)).
$$
\end{example}

\subsection{Main results}

In this section, we present the main results of our manuscript. The first result concerns a higher regularity estimate for bounded viscosity solutions along the free boundary (or critical) points.

\begin{theorem}[{\bf Higher regularity estimates}]\label{Hessian_continuity}
Let $u \in C^0(B_1)$ be a viscosity solution of problem \eqref{pobst}. Assume that $f(|x|, t) \simeq |x - x_0|^{\alpha} t_+^m$ with \eqref{EqHomog-f} in force. Then, for any point $x_0 \in B_{1/2} \cap \partial \{ u > 0 \}$, we have
\begin{equation}\label{Higher Reg}
\sup_{x \in B_r(x_0)} u(x) \leq \mathrm{C} \cdot r^{\frac{4 + \alpha}{3 - m}}
\end{equation}
for $r \in (0, 1/2)$, where $\mathrm{C} > 0$ is a universal constant\footnote{Throughout this manuscript, a universal constant is the one depending only on universal parameters, namely $m$, $\alpha$, and the dimension.}.
\end{theorem}

\begin{remark}\label{Remark1.1} 
Our results remain significant even in the absence of strong absorption, i.e., when $m = 0$. In this case, we obtain sharper estimates than those associated with the zero-obstacle problem (cf. \cite{RTU15}), as well as those related to the $C^{1, \frac{1}{3}}$-regularity conjecture (cf. \cite{daSRosSal19}), since
$$
\frac{4 + \alpha}{3} > \frac{4}{3} \quad \Leftrightarrow \quad \alpha > 0.
$$
Moreover, even when dealing with a singular weight, i.e., $\alpha < 0$ (recall the assumption in \eqref{EqHomog-f}), we still achieve improved estimates, since
$$
\frac{4 + \alpha}{3 - m} > \frac{4}{3} \quad \Leftrightarrow \quad \alpha > -\frac{4}{3} m.
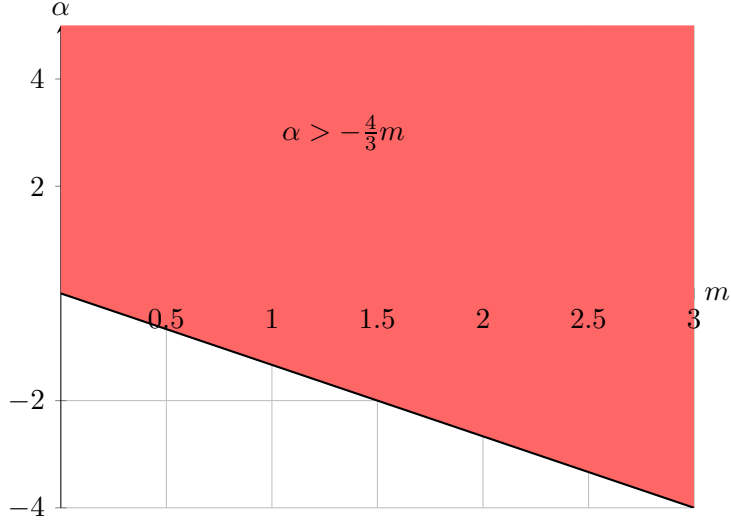
$$

\begin{tikzpicture}
  \begin{axis}[
    axis lines=middle,
    xlabel={$m$},
    ylabel={$\alpha$},
    ymin=-4, ymax=5,
    xmin=0, xmax=3,
    samples=200,
    domain=0:3,
    width=10cm, height=8cm,
    grid=both,
    xlabel style={right},
    ylabel style={above}
  ]
    \addplot [
      name path=A,
      domain=0:3,
      samples=200,
      thick
    ] {-4/3 * x};
    
    \path[name path=B] (axis cs:0,5) -- (axis cs:3,5);
    
    \addplot [
      fill=red!60,
      opacity=0.3
    ] fill between[of=A and B];

    \node at (axis cs:1,3) [anchor=west] {$\alpha > -\frac{4}{3}m$};
  \end{axis}
\end{tikzpicture}

\end{remark}

\begin{remark}
Another important point to emphasize in Theorem \ref{Hessian_continuity} is the following: Let $x_0 \in \Omega$, where $0 \in \Omega$, and consider the linear transformation
$$
\mathrm{T}_{x_0}: \mathbb{R}^n \rightarrow \mathbb{R}^n \quad \text{such that} \quad \mathrm{T}_{x_0}(x) = x - x_0.
$$
Since the diffusion model
\begin{equation}\label{EqWeight}
\Delta_{\infty} u(x) = \mathrm{f}(|x|) \quad \text{in} \quad \Omega    
\end{equation}
is invariant under the transformation $z=\mathrm{T}_{x_0}(x) = x-x_0$, where
$$
\mathrm{f} \in C^0([0, \infty)) \quad \text{and} \quad  |\mathrm{f}(|z|)| \le \mathrm{c}_n |z|^{\alpha}, \quad \text{for} \quad \alpha > 0,
$$
then a viscosity solution belongs to $C^{\frac{4 + \alpha}{3}}$ at $x_0$.

In particular, this behavior occurs when the weight function is a multiple of the distance to a fixed closed set. Specifically, let $\mathrm{F} \subset\subset \Omega$ be a fixed closed set, and let $\mathrm{f}(|x|) = \mathrm{dist}^{\alpha}(x, \mathrm{F})$ for a given $\alpha > 0$. Then, viscosity solutions to \eqref{EqWeight} belong to $C^{\frac{4 + \alpha}{3}}$ along the set $\mathrm{F}$. Moreover, 
$$
\sup_{B_r(x_0) \atop{x_0 \in \mathrm{F}}} |u(x) - u(x_0)| \leq \mathrm{C}(\text{universal}) \cdot r^{\frac{4 + \alpha}{3}}.
$$

As an illustration, consider the function
$$
\Psi_0(x) = \left(\frac{3^4}{(4 + \alpha)^3(1 + \alpha)}\right)^{\frac{1}{3}}\left(|x - x_0| - r_0\right)_{+}^{\frac{4 + \alpha}{3}},
$$
which is a viscosity solution of 
$$
\Delta_\infty \Psi_0(x) = \mathrm{dist}^{\alpha}\left(x, \overline{B_{r_0}(x_0)}\right) \quad \text{in} \quad \mathbb{R}^n \setminus B_{r_0}(x_0) \quad \text{and} \quad \Psi_0 \in C^{\frac{4 + \alpha}{3}}(\mathbb{R}^n).
$$
\end{remark}

We also establish the sharp growth rate for the gradient of solutions near free boundary points. For this purpose, for any point $z \in \{ u > 0\} \cap B_{1/2}$, let us define $\zeta(z) \in \partial \{ u > 0\}$ such that 
$$
    |z - \zeta(z)| = \text{dist} \left( z, \partial \{ u > 0\} \right).
$$

\begin{corollary}[{\bf Gradient growth}]\label{gradiente control}
Let $u \in C^0(B_1)$ be a viscosity solution to \eqref{pobst} in $B_1$, and suppose that \eqref{EqHomog-f} holds. Then, $u$ is locally Lipschitz continuous. Moreover, for any point 
$x_0 \in \{ u > 0 \} \cap B_{1/2}$, such that $f(|x|, t) \simeq |x - \zeta(x_0)|^{\alpha} t_+^m$, we have
\[
|D u(x_0)| \leq \mathrm{C}_0 \cdot \mathrm{dist}(x_0, \partial \{ u > 0 \})^{\frac{1+\alpha+m}{3-m}},
\]
where $\mathrm{C}_0 > 0$ is a universal constant.
\end{corollary}

From now on, the critical set of solutions  consists of 
$$
\mathcal{Z}^{\prime}_0(u, \Omega) = \{ x \in \Omega : u(x) = | Du(x)| = 0\}. 
$$

Next, we establish a weak geometric property, specifically a non-degeneracy estimate. More precisely, a viscosity solution $u$ departs from the free boundary at a rate proportional to $r^{\frac{4+\alpha}{3-m}}$.

\begin{theorem}[{\bf Non-degeneracy at critical points}]\label{NãoDeg}
Let $m \in [0, 3)$ and $u \in C^0(B_1)$ be a viscosity solution to problem \eqref{pobst}. Then, there exists $r^{\ast} > 0$ such that for every critical point $x_0 \in B_1$ and for all $r \in (0, r^{\ast})$ such that $B_r(x_0) \subset B_1$, we have
$$
\sup_{\partial B_r(x_0)}  u(x) \geq \left(\frac{(3-m)^4}{(4+\alpha)^3(1+\alpha+m)}\right)^{\frac{1}{3-m}} \cdot r^{\frac{4+\alpha}{3-m}}.
$$
\end{theorem}

We will now explore the various aspects of model \eqref{pobst} by examining a simpler one-dimensional profile.

\begin{example}[{\bf A one-dimensional profile}]\label{Examp01}
Consider the one-dimensional function
\[
u_i(x_1, \ldots, x_n) = \left(\frac{(3-m)^4}{(4+\alpha)^3(1+\alpha+m)}\right)^\frac{1}{3-m} |x_i|^{\frac{4+\alpha}{3-m}} \quad \text{for} \quad x \in B_1.
\]
We have that $u_i$ satisfies, in the viscosity sense,
$$
\Delta_\infty u_i(x) = 
\left\{
\begin{array}{lcl}
   |x_i|^{\alpha} (u_i)_+^m(x) \lesssim  |x|^{\alpha} u_+^m(x)  & \mbox{in} \quad B_1, &\mbox{if} \quad \alpha, m \neq 0, \\
    |x_i|^{\alpha} \lesssim  |x|^{\alpha} & \mbox{in} \quad B_1, & \mbox{if} \quad \alpha \neq 0, \,\,m=0, \\
    (u_i)_+^m(x) & \mbox{in} \quad B_1, & \mbox{if} \quad \alpha=0, \,\, m \neq 0, \\
    1 & \mbox{in} \quad B_1, & \mbox{if} \quad \alpha, m=0,
\end{array}
\right.
$$
and
{\scriptsize{\[
\mathcal{S}_u(\Omega) = \{x \in B_1: |D u| = 0\} = \{x \in B_1: \,\, x_i = 0\} \quad \text{with} \quad \mathscr{H}_{\text{dim}}(\mathcal{S}_u(\Omega)) = 1.
\]}}
where \( \mathscr{H}_{\text{dim}} \) denotes the Hausdorff dimension. In particular, we have that $u$ belongs to $C_{\text{loc}}^{1,\frac{1+\alpha+m}{3-m}}(B_1)$. Moreover, if the conditions on the exponents in \eqref{EqHomog-f} hold, then
$$
1 + \alpha + m > 1 - \frac{4}{3}m + m = 1 - \frac{m}{3} > 0.
$$
\end{example}

\begin{remark}
We must highlight some consequences of the previous example. Let $\beta = \frac{1 + \alpha + m}{3 - m}$. Then, we observe that along critical points
$$
\left\{
\begin{array}{ccccc}
  \beta \in (0, 1) & \Leftrightarrow & \alpha < 2(1 - m) & \Rightarrow & u \,\,\text{belongs to } \,\,C^{1, \beta} \,\, \text{at the origin}, \\
  \beta = 1 & \Leftrightarrow & \alpha = 2(1 - m) & \Rightarrow & u \,\,\text{belongs to } \,\,C^{1, 1} \,\, \text{at the origin}, \\
  \beta > 1 & \Leftrightarrow & \alpha > 2(1 - m) & \Rightarrow & u \,\,\text{belongs to } \,\,C^{2} \,\, \text{at the origin}.
\end{array}
\right.
$$
In particular, weak solutions enjoy classical estimates along the set of critical points, provided $\alpha > 2(1 - m)$ (see, e.g., Example \ref{Examp01}).
\end{remark}

\begin{example}
Another relevant model for \eqref{pobst} is the Hardy-H\'{e}non-type problem with general weights, absorption, and noise terms:
$$
\displaystyle \Delta_{\infty} u (x) =  \sum_{i=1}^{l_0}  \mathrm{dist}^{\alpha_i}(x, \mathrm{F}_i) (u)_{+}^{m_i}(x) + g_i(|x|) \quad \text{in} \quad B_1,
$$
where $\mathrm{F}_i \subset B_1$ are disjoint closed sets, $m_i \in [0, 3)$, $\alpha_i > -\frac{4}{3}m_i$, and
$$
\displaystyle \limsup_{|x| \to 0} \frac{g_i(|x|)}{\mathrm{dist}(x, \mathrm{F}_i)^{\kappa_i}} = \mathfrak{L}_i \in [0, \infty) \quad \text{for some} \quad \kappa_i \ge 0.
$$
In this context, viscosity solutions belong to $ C_{\text{loc}}^{\displaystyle \min_{1 \le i \le l_0} \left\{\frac{4 + \alpha_i}{3 - m_i}, \frac{4 + \kappa_i}{3}\right\}}$ along the sets $\mathrm{F}_i \cap \mathcal{S}_u(B_1)$.
\end{example}

\subsection{Further results and implications}

Before presenting some applications of our findings, we will analyze radial solutions to our model equation to understand the classification results of global profiles, i.e., Liouville-type results.

\begin{example}[\textbf{A radial profile}]\label{RadialEx} 
Consider the function $\Psi: B_1 \rightarrow \mathbb{R}$ given by
$$
\Psi(x) = \mathrm{C}_{\alpha, m}|x|^{\frac{4 + \alpha}{3 - m}},
$$
where $\mathrm{C}_{\alpha, m} > 0$ is a constant that will be chosen \textit{a posteriori}. Since $\Psi$ is radial, we have
{\scriptsize{
\begin{eqnarray*}
\Delta_\infty \Psi(x) &=& \left(\Psi'(x)\right)^2 \Psi''(x) \\[0.2cm]
&=& \left[\mathrm{C}_{\alpha, m} \left(\frac{4 + \alpha}{3 - m}\right)|x|^{\frac{4 + \alpha}{3 - m} - 1}\right]^2 \left[\mathrm{C}_{\alpha, m} \left(\frac{4 + \alpha}{3 - m}\right) \left(\frac{1 + \alpha + m}{3 - m}\right)|x|^{\frac{4 + \alpha}{3 - m} - 2}\right] \\[0.2cm]
&=& \frac{(4 + \alpha)^3}{(3 - m)^4}(1 + \alpha + m)\mathrm{C}_{\alpha, m}^3 |x|^{3\left(\frac{4 + \alpha}{3 - m}\right) - 4}.
\end{eqnarray*}
}}
Thus, if we choose
$$
\mathrm{C}_{\alpha, m} = \left(\frac{(3 - m)^4}{(4 + \alpha)^3(1 + \alpha + m)}\right)^{\frac{1}{3 - m}},
$$
we conclude that $\Psi$ is a viscosity solution of
$$
\Delta_\infty \Psi(x) = |x|^\alpha \Psi_+^m(x) \quad \text{in} \quad B_1
$$
and
\[
\mathcal{S}_\Psi(\Omega) = \{x \in B_1: |D \Psi| = 0\} = \{x \in B_1: x = 0\} \quad \text{with} \quad \mathscr{H}_{\text{dim}}(\mathcal{S}_\Psi(\Omega)) = 0.
\]
In particular, we have that $u$ belongs to $C_{\text{loc}}^{1, \frac{1 + \alpha + m}{3 - m}}(B_1)$.
\end{example}

Liouville-type theorems are well-known in the context of elliptic PDEs and have played an important role in the modern theory of mathematical analysis due to their applications in nonlinear equations, free boundary problems, and differential geometry, to mention just a few topics. 

The main purpose of this subsection is to prove that a global solution to 
$$
\Delta_{\infty} u(x) = |x|^{\alpha} u^m(x) \quad \text{in} \quad \mathbb{R}^n
$$
must grow faster than $\mathrm{c}_0|x|^{\frac{4+\alpha}{3-m}}$ as $|x| \to \infty$ for a suitable constant $\mathrm{c}_0 > 0$, unless it is identically zero.

As an application of our findings, we will present some Liouville-type results for a class of Hardy-H\'{e}non-type equations (cf. \cite{BisVo23} for related results). Specifically,

\begin{theorem}[{\bf Liouville-type result $\mathrm{I}$}]\label{Liouville I}  
Let $u$ be a non-negative entire viscosity solution to
$$
\Delta_{\infty} u(x) = |x|^{\alpha} u^m(x) \quad \text{in} \quad \mathbb{R}^n
$$
with $u(0) = 0$. Suppose that
\begin{equation}\label{hypothesis}
\displaystyle \lim_{|x| \to \infty} \frac{u(x)}{|x|^{\frac{4+\alpha}{3-m}}} = 0,
\end{equation}
then $u \equiv 0$.
\end{theorem}

\begin{theorem}[{\bf Liouville-type result $\mathrm{II}$}]\label{Liouville II} Given $r>0$ and $x_0 \in \mathbb{R}^n$, let $u$ be a non-negative entire viscosity solution to
$$
\Delta_{\infty} u(x) = \mathfrak{h}_{r, x_0}(|x|) u^m(x) \quad \text{in} \quad \mathbb{R}^n, 
$$
where the weight $\mathfrak{h}_{r, x_0}(|x|) \lesssim (|x-x_0|-r)_+^{\alpha}$. Suppose that
\begin{equation}\label{Liouville II-H}
\displaystyle \limsup_{|x| \to \infty} \frac{u(x)}{|x|^{ \frac{4+\alpha}{3-m}}} < \left(\frac{(3-m)^4}{(4+\alpha)^3(1+\alpha+m)}\right)^{\frac{1}{3-m}},
\end{equation}
then $u \equiv 0$.
\end{theorem}

As a result of Theorems \ref{Hessian_continuity} and \ref{NãoDeg}, we obtain a positive density result for the non-coincidence set.

\begin{corollary}[{\bf Uniform Positive Density}]\label{positive density}
Let \( u \) be a weak solution of \eqref{pobst}, and let \( x_0 \in \{u > 0\} \cap B_{1/2} \). Then, there exists a universal constant \( \mathrm{c}_0 > 0 \) such that 
$$ 
\frac{\mathcal{L}^n(\{u > 0\} \cap B_r(x_0))}{\mathcal{L}^n( B_r(x_0))} \geq \mathrm{c}_0.
$$
\end{corollary}

As another consequence of the growth rate and the non-degeneracy property, we obtain the porosity of the zero-level set.

\begin{definition}[{\bf Porous Set}]
    A set \( \mathcal{S} \subset \mathbb{R}^n \) is said to be \textit{porous} with porosity \( \delta_{\mathcal{S}} > 0 \) if there exists \( R > 0 \) such that
\[
\forall \, x \in \mathcal{S}, \; \forall \, r \in (0, R), \,\, \exists\,\, y \in \mathbb{R}^n \; \text{such that} \; B_{\delta_{\mathcal{S}} r}(y) \subset B_r(x) \setminus \mathcal{S}.
\]
\end{definition}

We observe that a porous set with porosity \( \delta_{\mathcal{S}} > 0 \) satisfies 
$$
\mathscr{H}_{\text{dim}}(\mathcal{S}) \leq n - \mathrm{c}_0 \delta_{\mathcal{S}}^n 
$$
where \( \mathrm{c}_0 = \mathrm{c}_0(n) > 0 \) is a dimensional constant. In particular, a porous set has Lebesgue measure zero (see, for example, \cite{Zaj87}).

\vspace{0.3cm}

\begin{corollary}[{\bf Porosity}]\label{Hausdorff dimension}
Let \( u \) be a non-negative weak solution of \eqref{pobst}. Then, the level set \( \partial \{ u > 0 \} \cap B_r \) is porous with a porosity constant independent of \( r > 0 \).
\end{corollary}

It is important to note that the proofs of Corollaries \ref{positive density} and \ref{Hausdorff dimension} are analogous to those presented in \cite[Corollary 4.5 and Corollary 4.7]{daSRosSal19} and \cite[Corollary 5.3 and Corollary 5.5]{daSSal18}. As a result, we have chosen not to include the full demonstrations here to avoid unnecessary repetition of these arguments.

\subsection{The Borderline Scenario: A Strong Maximum Principle}

In this section, we will turn our attention to the critical equation obtained as \( m \to 3^{+} \), namely,
\begin{equation}\label{Crit_Eq}
\Delta_{\infty}\,u(x) = |x|^{\alpha}u^{3}(x) \quad \text{in} \quad \Omega.
\end{equation}

Observe that this equation is critical, as all the regularity estimates established so far degenerate when \( m \to 3 \). Moreover, we can rewrite equation \eqref{Crit_Eq} as
$$
\Delta_{\infty}\, u(x) = |x|^{\alpha}u^{\kappa} u^{3-\kappa}(x)
$$
for any \( \kappa>0 \). In particular, we obtain from Theorem \ref{Hessian_continuity} that if \( x_0 \in B_1 \) satisfies \( u(x_0)=0 \), then \( D^{|\tau|} u(x_0) =0 \) for all multi-indices \( \tau \in \mathbb{N}^n \). Thus, any vanishing point is a zero of infinite order. Finally, under the suitable (and strong) assumption that \( u \) is an analytic function, one could conclude that \( u \equiv 0 \).

In this context, utilizing the Maximum-Minimum Principle, we will demonstrate that a solution to \eqref{Crit_Eq} cannot vanish at interior points.

\begin{theorem}[{\bf Strong Maximum Principle}]
Let \( u \in C^{0}(\Omega) \) be a viscosity solution to \eqref{Crit_Eq}. Suppose that there exists an interior vanishing point \( x_0 \in \Omega \), i.e., \( u(x_0) =0 \). Then, \( u \equiv 0 \) in \( \Omega \).
\end{theorem}

\begin{proof}
The proof of this result follows from \cite[Remark 6.4]{BhatMoh12}. We will present it here to guide the reader. Consider \( \mathcal{O}^+ = \{x \in \Omega \,|\, u(x) > 0\} \). If \( \mathcal{O}^+ \) is nonempty, we obtain that 
$$
\Delta_\infty u \ge 0 \quad \text{in} \quad \mathcal{O}^+.
$$
By the Maximum Principle (Lemma \ref{MP}), and since there exists \( x_0 \in \Omega \) such that \( u(x_0) = 0 \), we conclude that
$$
u(x) \le \sup_{\mathcal{O}^+} u = \sup_{\partial \mathcal{O}^+} u = 0,
$$
i.e., \( u \equiv 0 \) in \( \mathcal{O}^+ \), which is a contradiction.

Similarly, if we consider \( \mathcal{O}^- = \{x \in \Omega \,|\, u(x) < 0\} \) and suppose that \( \mathcal{O}^- \) is nonempty, we have 
$$
\Delta_\infty u \le 0 \quad \text{in} \quad \mathcal{O}^-.
$$
By the Minimum Principle (Lemma \ref{MP}), and since there exists \( x_0 \in \Omega \) such that \( u(x_0) = 0 \), we conclude that
$$
0 = \inf_{\partial \mathcal{O}^-} u = \inf_{\mathcal{O}^-} u \le u(x).
$$
Hence, \( u \equiv 0 \) in \( \mathcal{O}^- \), which is also a contradiction. In conclusion, if there exists an interior vanishing point \( x_0 \in \Omega \), we obtain that \( u \equiv 0 \) in \( \Omega \).
\end{proof}

\subsection*{Acknowledgments}

J.V. da Silva has received partial support from CNPq-Brazil under Grant No. 307131/2022-0 and FAEPEX-UNICAMP 2441/23 Editais Especiais - PIND - Projetos Individuais (03/2023).  G.S. S\'{a} expresses gratitude for the PDJ-CNPq-Brazil (Postdoctoral Scholarship - 174130/2023-6).

\section{A relevant review of the literature}

\subsection{Hardy-H\'{e}non-type models}

In the sequel, we will discuss several semilinear elliptic models arising in astrophysics and various diffusion phenomena.

Historically, the H\'{e}non equations emerged in 1973 when the French mathematician and astronomer Michel H\'{e}non, in \cite{Henon73}, utilized the concentric shell model to numerically explore the stability of spherical stellar systems in equilibrium against spherical perturbations. He examined a type of semilinear problem. By way of explanation, it can be inferred that the existence of stationary stellar dynamics models, as seen in \cite{Batt77}, \cite{BFH86}, and \cite{LS95}, corresponds to the solvability of the semilinear equation
\[
-\Delta \mathrm{U}(x) = \mathrm{f}(|x|, \mathrm{U}(x)) \text{ in } \mathbb{R}^3,
\]
which, in the scenario where 
$$\mathrm{f}(|x|, \mathrm{U}(x)) = |x|^\alpha |\mathrm{U}(x)|^{p-2}\mathrm{U}(x),\quad  \alpha > 0,\,\,\, p > 2,
$$
transforms into the now well-known H\'{e}non equation,
\[
-\Delta \mathrm{U}(x) = |x|^\alpha |\mathrm{U}(x)|^{p-2}\mathrm{U}(x) \text{ in } \mathbb{R}^3,
\]
where the weight \( \mathfrak{h}(x) = |x|^\alpha \) represents a black hole situated at the center of the cluster, with its absorption intensity increasing as \( \alpha \) grows.

Currently, such elliptic models, notably the well-known H\'{e}non equations \cite{Henon73}, play a significant role in astrophysics and serve as standard models for several diffusion phenomena. From a mathematical perspective, the problem 
\begin{equation}\label{EqHenon}
\left\{
\begin{array}{rclcl}
-\Delta u(x) & = & |x|^{\alpha} |u|^{p-2} u & \text{in } & \Omega, \\
u(x) & = & 0 & \text{on } & \partial \Omega,
\end{array}
\right.    
\end{equation}
exhibits a fascinating structure, and numerous results have been established so far, some of which will be outlined in the sequel. Notably, in the context of \eqref{EqHenon}, Ni noted in \cite{Ni82} that the inclusion of the weight $\mathfrak{h}(x) = |x|^{\alpha}$ alters the outcomes of the Poho\v{z}aev identity, leading to a novel critical exponent, specifically \(p_{\alpha} = \frac{2(n + \alpha)}{n - 2}\), for the existence of classical solutions. Furthermore, symmetry breaking, asymptotics, and single-point concentration profiles at the boundary of the least-energy solutions, as \(\alpha \to \infty\), are discussed in \cite{BW06}, \cite{CPY09}, and \cite{SWS02}. Finally, in \cite{dosSP16}, the authors studied the concentration profiles of various types of symmetric positive solutions.

Recall that for \(p>1\), the elliptic model 
\begin{equation}\label{Eq-H-H}
\left\{
\begin{array}{rclcl}
-\Delta u(x) & = & |x|^{\alpha} u^p(x) & \text{in } & \Omega, \\
u(x) & > & 0 & \text{in } & \Omega, \\
u(x) & = & 0 & \text{on } & \partial \Omega,
\end{array}
\right.    
\end{equation}
is referred to as Hardy-type if \(\alpha < 0\) due to its connection to the Hardy-Sobolev inequality (see \cite{CotLa19}), and as H\'{e}non-type if \(\alpha > 0\).

Finally, using variational methods, one establishes the existence of a nontrivial solution to \eqref{Eq-H-H} in \(H_0^1(\Omega)\) (for \(\Omega \subset \mathbb{R}^n\) being an arbitrary bounded domain) provided that 
\[
2 < p+1 < \frac{2(n - |\alpha|)}{n-2},
\]
through an application of the seminal Caffarelli-Kohn-Nirenberg estimates (see \cite{CKN84}). Furthermore, by employing a generalized Poho\v{z}aev-type identity, one proves the non-existence of nontrivial solutions in star-shaped domains provided that 
\[
0 \geq \alpha > -n \quad \text{and} \quad p+1 = \frac{2(n - |\alpha|)}{n-2}.
\]

Therefore, the simplest model problem described by \eqref{EqHenon} (for non-negative viscosity solutions) serves as our initial impetus for studying Hardy-H\'{e}non-type equations driven by $\infty$-Laplacian with general weights, which may have a singular signature (see, Remark \ref{Remark1.1}).

\subsection{$\infty$-Laplacian and its Existence/Regularity Theories}

Historically, a comprehensive investigation of problems involving the \(\infty\)-Laplacian operator dates back to the groundbreaking contributions of Aronsson in \cite{Aronsson67} and \cite{Aronsson68}. Mathematically, the core focus of Aronsson's research was to address the following pivotal issue: given a bounded domain \(\Omega \subset \mathbb{R}^n\) and a Lipschitz function \(g: \partial \Omega \rightarrow \mathbb{R}\), find its optimal Lipschitz extension, denoted by \(\mathcal{G}\), such that it coincides with \(g\) on the boundary and satisfies: 
$$
\text{for any } \Omega' \subset \Omega, \text{ if } \mathcal{G} = h \text{ on } \partial \Omega', \text{ then } \|\mathcal{G}\|_{C^{0, 1}(\Omega')} \leq \|h\|_{C^{0, 1}(\Omega')}.
$$
Such a function \(\mathcal{G}\) is now referred to as an \textit{Absolutely Minimizing Lipschitz Extension} of \(g\) in \(\Omega\) (AMLE for short).

Afterward, Jensen demonstrated the following enlightening fact in \cite{Jensen93}:
$$
u \quad \text{is an AMLE} \qquad \Leftrightarrow \qquad -\Delta_{\infty} u(x) = 0 \quad \text{in the viscosity sense}.
$$
In other words, the \(\infty\)-Laplacian operator governs the Euler–Lagrange equation associated with this \(L^{\infty}\)-Lipschitz minimization problem. In particular, Jensen proved that a viscosity solution to the \(\infty\)-Laplacian subject to a given Dirichlet boundary condition is unique (cf. \cite{ArmSmart10} for an elementary proof of Jensen's theorem).

Moreover, infinity harmonic functions and their generalizations arise in several other contexts of pure and applied mathematics; see \cite{BEJ08} for an interesting survey. In particular, the value of a "Tug-of-War" game corresponds to an infinity harmonic profile; see e.g., \cite{PSSW09} for an enlightening work.

Nowadays, the regularity theory for the \(\infty\)-Laplacian operator is a prominent and challenging research theme in the modern context of nonlinear PDEs and related topics. A longstanding conjecture in this area asserts that viscosity solutions of 
$$ 
-\Delta_{\infty} u(x) = 0 \quad \text{in} \quad \Omega 
$$ 
belong to the class \(C^{1,\frac{1}{3}}\), which remains unverified despite some recent exciting advancements (cf. \cite{daSRosSal19} for a type of infinity-dead core problem and \cite{RTU15} for the infinity-obstacle problem). 

By way of illustration, Aronsson's family of explicit solutions (see \cite{Aronsson84}) is given by 
$$
u(x) = u(x_1, \cdots , x_n) = a_1|x_1|^{\frac{4}{3}} + \cdots + a_n|x_n|^{\frac{4}{3}} \quad (\text{for } \sum_{i=1}^{n} a_i^3 = 0)
$$
provides what we should expect, namely that the first derivatives of \(u\) are \(\alpha\)-H\"{o}lder continuous with a sharp exponent \(\alpha = \frac{1}{3}\). Moreover, their second derivatives fail to exist along the coordinate axes.

As an illustrative example, in the \(2\)-D scenario (and even in higher dimensions, see \cite{Yu06}), the result below allows us to conclude that there are solutions that do not have second continuous derivatives.

\begin{theorem}[{\bf Aronsson \cite{Aronsson68}}] 
Suppose that \(u \in C^2(\Omega)\), where \(\Omega \subset \mathbb{R}^n\) is a domain. If 
$$
-\Delta_{\infty} u(x) = 0 \quad \text{in} \quad \Omega,
$$
then \(|D u| \neq 0\) in \(\Omega\), except when \(u\) is a constant.

\end{theorem}

Regarding regularity estimates, the most precise results to date are attributed to Evans and Savin, who demonstrate in \cite{ES08} that infinity-harmonic functions in the plane belong to the class \(C^{1,\alpha_0}\), building on Savin’s breakthrough in \cite{Savin05} (note that the optimal \(\alpha_0 \in (0, 1)\) remains unknown even in \(2\)-D), and to Evans and Smart, who recently achieved differentiability everywhere, regardless of the dimension in \cite{ESmart11}.

The theory of inhomogeneous \(\infty\)-Laplacian equations 
$$
-\Delta_{\infty} u(x) = f(x) \quad \text{in} \quad \Omega
$$ 
is more recent and challenging. In this context, Lu and Wang in \cite{LuWang08} established the existence and uniqueness of continuous viscosity solutions to the Dirichlet problem:

\[
-\Delta_{\infty} u(x) = f(x) \quad \text{in } \Omega, \quad u = 0\quad \text{on } \partial \Omega,
\]
provided that the source function \(f\) does not change its sign (i.e., either \(\displaystyle \inf_{\Omega} f > 0\) or \(\displaystyle \sup_{\Omega} f < 0\)). Furthermore, uniqueness may fail if this condition is not satisfied (see, e.g., \cite[Appendix A]{LuWang08}). Despite the fact that Lipschitz estimates and differentiability also hold for a function whose \(\infty\)-Laplacian is bounded in the viscosity sense (see Lindgren's work \cite{Lind14}), no further regularity results are currently known for inhomogeneous equations.

\subsection{Elliptic diffusion models related to our research}

In the following, we will outline the connections and motivations between our problems \eqref{pobst} and several significant mathematical models arising in the context of free boundary problems and nonlinear diffusion processes.

As a primary example of a free boundary model related to our studies, we will mention the zero-obstacle problem addressed by Rossi \textit{et al.} in \cite{RTU15}, which involves examining a function that satisfies, in the viscosity sense,
\[
\Delta_{\infty} u = f(x) \quad \text{in } \quad  \{ u > 0 \} \quad \text{and} \quad u \geq 0 \quad \text{in } B_1.
\]
Alternatively, we can express the zero-obstacle problem as follows:
\begin{equation}\label{EqIOP}
\min \left\{ \Delta_{\infty} u(x) - f(x), u(x) \right\} = 0 \quad \text{in} \quad B_1    
\end{equation}
understood in the viscosity sense. In this context, the authors presented the following existence and regularity results.

\begin{theorem}[{\bf \cite[Theorem 3.1]{RTU15}}] Given a function $g \in C^0(\partial B_1)$, with $g > 0$, and $f$ satisfying
$$
0 < \mathfrak{m}_0 <f(x) \le \mathrm{M} < \infty,
$$
there exists a unique function $u \in C^0(\overline{B}_1)$ such that
$$
\left\{
\begin{array}{rclcl}
  \min \left\{ \Delta_{\infty} u(x) - f(x), u(x) \right\} & = & 0 & \text{in } & B_1,\\
   u(x)  & = & g(x) & \text{on } & \partial B_1
\end{array}
\right.
$$
in the viscosity sense. Moreover, if $f$ is uniformly Lipschitz continuous in $B_1$, then $u$ is locally Lipschitz continuous in $B_1$.
\end{theorem}

The authors also present the optimal $C^{1, \frac{1}{3}}$ regularity estimate for solutions of the $\infty$-obstacle problem along the free boundary.

\begin{theorem}[{\bf \cite[Theorem 3.3]{RTU15}}] Let $u$ be a solution to \eqref{EqIOP} and let $x_0 \in \partial \{ u > 0 \}$ be an arbitrary free boundary point. Then,
\[
\sup_{B_r(x_0)} u(x) \leq \mathrm{C} \, r^{\frac{4}{3}}.
\]
for some constant $\mathrm{C}>0$ that depends only on the problem data.
\end{theorem}

Additionally, they also prove that viscosity solutions leave the free boundary in a $C^{1, \frac{1}{3}}$-fashion.

\begin{theorem}[{\bf \cite[Theorem 3.5]{RTU15}}]
Let $u$ be a viscosity solution to \eqref{EqIOP} and let $y_0 \in \{ u > 0 \}$ be an arbitrary point in the closure of the non-coincidence set. Then,
\[
\sup_{B_r (y_0)} u \geq  \mathrm{c}(\mathfrak{m}_0) \, r^{\frac{4}{3}},
\]
for some universal constant $\mathrm{c} > 0$.
\end{theorem}

The second relevant free boundary model that we must highlight concerns the $\infty$-dead core problem. Specifically, in \cite{ALT16}, Ara\'{u}jo \textit{et al.} focused on investigating reaction-diffusion models governed by the $\infty$-Laplacian operator. For $\lambda > 0$, $0 \leq \gamma < 3$, and $0 < \phi \in C^0(\partial \Omega)$, let $\Omega \subset \mathbb{R}^n$ be a bounded open domain, and define
\begin{equation}\label{EqDeadCore}
    \mathcal{L}^{\gamma}_{\infty} v := \Delta_{\infty} v - \lambda \cdot (v_+)^{\gamma} = 0 \quad \text{in} \quad \Omega \quad \text{and} \quad v = \phi \quad \text{on} \quad \partial \Omega.
\end{equation}
Here, the operator $\mathcal{L}^{\gamma}_{\infty}$ represents the $\infty$-diffusion operator with $\gamma$-strong absorption, and the constant $\lambda > 0$ is referred to as the \textit{Thiele modulus}, which controls the proportion between the reaction rate and the diffusion-convection rate (cf. \cite{Diaz85}, see also \cite{DT20} for related results).

We must also emphasize that a significant feature in the mathematical representation of equation \eqref{EqDeadCore} is the potential occurrence of plateaus, namely subregions, \textit{a priori} unknown, where the function becomes identically zero (see \cite{daSRosSal19}, \cite{daSSal18}, and \cite{Diaz85} for corresponding quasi-linear dead-core problems).

After proving the existence and uniqueness of viscosity solutions via Perron's method and a Comparison Principle (see \cite[Theorem 3.1]{ALT16}), the main result in the manuscript \cite{ALT16} guarantees that a viscosity solution is point-wisely of class \( C^{\frac{4}{3-\gamma}} \) along the free boundary of the non-coincidence set, i.e., \( \partial \{ u > 0 \} \) (see \cite[Theorem 4.2]{ALT16}). This implies that the solutions grow precisely as \( \mathrm{dist}^{\frac{4}{3-\gamma}} \) away from the free boundary. Additionally, by utilizing barrier functions, the authors demonstrate that such an estimate is sharp in the sense that \( u \) separates from its coincidence region precisely as \( \mathrm{dist}^{\frac{4}{3-\gamma}} \) (see, e.g., \cite[Theorem 6.1]{ALT16}).

Furthermore, the authors presented some Liouville-type results. Specifically, if \( u \) is an entire viscosity solution to
$$
\Delta_{\infty} u(x) = \lambda u^{\gamma}(x) \quad \text{in} \quad \mathbb{R}^n,
$$
with \( u(0) = 0 \), and \( u(x) = \text{o}\left(|x|^{\frac{4}{3-\gamma}}\right) \), then \( u \equiv 0 \) (see \cite[Theorem 4.4]{ALT16}).

Furthermore, they proved a more refined quantitative version of the previous result. Specifically, if \( u \) is an entire viscosity solution to
$$
\Delta_{\infty} u (x) = \lambda u^{\gamma}(x) \quad \text{in} \quad \mathbb{R}^n,
$$
such that
$$
\limsup_{|x| \to \infty} \frac{u(x)}{|x|^{\frac{4}{3-\gamma}}} < \left(\frac{\lambda(3-\gamma)^4}{4^3(1+\gamma)}\right)^{\frac{1}{3-\gamma}},
$$
then \( u \equiv 0 \) (see \cite[Theorem 5.1]{ALT16}).

We should also mention the work \cite{BisVo23}, where the authors establish Liouville-type results for a nonlinear equation involving the $\infty$-Laplacian with a gradient, expressed as
\[
\Delta^{\gamma}_{\infty} u + q(x) \cdot D u |D u|^{2-\gamma} + f(x,u) = 0 \quad \text{in} \quad \mathbb{R}^d,
\]
where \( \gamma \in [0,2] \) and \( \Delta^{\gamma}_{\infty} \) is a \((3-\gamma)\)-homogeneous operator related to the $\infty$-Laplacian. Specifically,
$$
\Delta^{\gamma}_{\infty} u := \frac{1}{|D u|^\gamma} \sum_{i,j=1}^{d} D_{i} u \, D_{i j} u \, D_{j} u.
$$
Moreover, note that \( \Delta^{\gamma}_{\infty} u \) corresponds to the classical $\infty$-Laplacian when \( \gamma = 0 \), while it becomes the normalized $\infty$-Laplacian for \( \gamma = 2 \) (cf. \cite{LY12} for related topics).

In this context, assuming that 
$$
\liminf_{|x| \to \infty} \lim_{s \to 0} \frac{f(x,s)}{s^{3-\gamma}} > 0
$$ 
and that \( q \) is a continuous function vanishing at infinity, the authors constructed a positive, bounded solution to the equation. Furthermore, if \( s \mapsto \frac{f(x,s)}{s^{3-\gamma}} \) is a decreasing function, they established uniqueness by refining the sliding method for the \( \infty \)-Laplacian operator with a nonlinear gradient.

Finally, as a matter of motivation, we should mention Teixeira's work in \cite{Tei22}, where the author studied regularity estimates at interior stationary points of solutions to \( p \)-degenerate elliptic equations in an inhomogeneous medium:
$$
\mathrm{div}\,\mathfrak{a}(x, \nabla u) = f(|x|, u) \lesssim \mathrm{c}_0|x|^{\alpha}|u|^{m}
$$
for some \( \mathrm{c}_0 > 0 \), \( \alpha \geq 0 \), and \( 0 \leq m < p - 1 \). In this scenario, the vector field \( \mathfrak{a}: B_1 \times \mathbb{R}^n \rightarrow \mathbb{R}^n \) is \( C^1 \)-regular in the gradient variable and satisfies the structural conditions: for every \( x,y \in B_1 \) and \( \xi, \eta \in \mathbb{R}^n \), the following holds:
\begin{equation} \label{condestr}
\left\{
\begin{array}{rclcl}
|\mathfrak{a}(x,\xi)| + |\partial_{\xi}\mathfrak{a}(x,\xi)||\xi| & \leq & \Lambda |\xi|^{p-1} & & \\
\lambda |\xi|^{p-2}|\eta|^2 & \leq & \langle \partial_{\xi}\mathfrak{a}(x,\xi)\eta,\eta \rangle & & \\
\displaystyle \sup_{x, y \in B_1 \atop{x \ne y, \,\,\,|\xi| \ne 0 }} \frac{|\mathfrak{a}(x,\xi)-\mathfrak{a}(y,\xi)|}{\omega(|x-y|)|\xi|^{p-1}} & \leq & \mathfrak{L}_0 < \infty, & &  \\
\end{array}
\right.
\end{equation}
where \( 2 < p < \infty \), \( 0 < \lambda \leq \Lambda < \infty \), \( \mathfrak{L}_0 > 0 \), and \( \omega: [0,\infty) \rightarrow [0,\infty) \) is a modulus of continuity for the coefficients of the vector field \( \mathfrak{a} \).

In this context, if the source term is bounded away from zero, Teixeira derives a quantitative non-degeneracy estimate, which indicates that solutions cannot be smoother than \( C^{p^{\prime}} \) at stationary points (see \cite[Proposition 5]{Tei22}). Furthermore, at critical points where the source vanishes, Teixeira proved higher-order regularity estimates that are sharp regarding the vanishing rate of the source term (see \cite[Theorem 3]{Tei22}).

As a final source of inspiration, we should also relate our problems to elliptic equations as follows:
\begin{equation}\label{EqMatukuma}
\mathrm{div}\left(\mathfrak{a}(|x|) |\nabla u|^{p-2} \nabla u\right) = \mathfrak{h}(|x|) f(u) \quad \text{in } \quad \Omega \subset \mathbb{R}^n, \quad  p > 1,
\end{equation}
where \(\mathfrak{a}, \mathfrak{h} : \mathbb{R}^+ \to \mathbb{R}^+\) represent radial profiles of class \(C^1(\mathbb{R}^+)\) and \(C^0(\mathbb{R}^+)\), respectively. As motivation, the celebrated Matukuma equation (or the Batt–Faltenbacher–Horst equation, see \cite{BFH86}) serves as a toy model for \eqref{EqMatukuma}. Additionally, strong absorption satisfies

\begin{itemize}
    \item[\textbf{(F1)}] \( f \in C^0(\mathbb{R}) \);
    \item[\textbf{(F2)}] \( f \) is non-decreasing on \( \mathbb{R} \), and \( f(t) > 0 \) if and only if \( t > 0 \).
\end{itemize}

Now, consider the model equation
\begin{equation}\label{ModelEq}
    \mathrm{div} \left( |x|^k | \nabla u|^{p-2} \nabla u \right) = |x|^{\alpha} f(u),
\end{equation}
where \( f(u) = u_{+}^{m} \) and \( m + 1 < k - \alpha < p \).

In this scenario, da Silva \textit{et al.} in \cite{daSdosPRS} addresses the following estimate for weak solutions to \eqref{ModelEq}:
$$
\sup_{x \in B_r(x_0)} u(x) \leq \mathrm{C} r^{1 + \frac{1 + \alpha + m - k}{p - 1 - m}},
$$
for a universal constant \( \mathrm{C} > 0 \), provided that \( x_0 \in B_1 \) is a free boundary point for \( u \), and \( f(|x|, t) \simeq |x - x_0|^{\alpha} t_+^m \), with \( \alpha + 1 + m > k \).

We summarize the sharp regularity estimates reported in the literature for several problems related to \eqref{pobst} in the table below.

\begin{table}[h]
\centering
\resizebox{\textwidth}{!}{
 \begin{tabular}{c|c|c|c}
{\bf Model problem} & {\bf Structural assumptions }  & {\bf Sharp regularity estimates} & \textbf{References} \\
\hline
 $\min \left\{ \Delta_{\infty} u(x) - f(x), u(x) \right\} = 0$ & $f \in L^{\infty}(B_1)$ \,\,\text{and}\,\,\,$u \geq 0$ & $C_
{\text{loc}}^{1, \frac{1}{3}}$ &  \cite{RTU15}  \\\hline
  $\Delta_{\infty} u(x) = \lambda u^{\gamma}(x) $ & $\lambda>0$ \,\,\text{and}\,\,$\gamma \in [0, 3)$ & $C_
{\text{loc}}^{\frac{4}{3-\gamma}}$ &  \cite{ALT16} and \cite{DT20} \\\hline
$\mathrm{div} \left( |x|^k | \nabla u|^{p-2} \nabla u \right) = |x|^{\alpha}u_+^m$ &  $0\leq m< p-1, \,\,\alpha + 1 + m > k \quad \text{and} \quad p \in (1, \infty)$ & $C_
{\text{loc}}^{\frac{p+\alpha-k}{p-1-m}}$ & \cite{daSdosPRS}\\
\hline
$\Delta_p u = f(x, u) \lesssim \mathrm{c}_0|x|^{\alpha}|u|^{m}$ & $\alpha \geq 0$ and $0\le m<p-1$ (for $p>2$) & $C_
{\text{loc}}^{1, \min\left\{\alpha_{\mathrm{H}}, \frac{p+\alpha+1}{p-1-m}\right\}^{-}}$ & \cite{Tei22} \\
\end{tabular}}
\end{table}

Particularly, we should stress that our work will improve some results previously presented, to some extent, via alternative strategies and techniques.

In conclusion, we emphasize that Teixeira's researches \cite{ALT16}, \cite{RTU15}, and \cite{Tei22} heavily inspire the current work, as well as we were motivated by the works of \cite{daSRosSal19}, \cite{daSSal18}, and \cite{daSdosPRS}. To the best of our knowledge, there is no investigation on this regularity topic for Hardy-H\'{e}non-type models driven by the $\infty$-Laplacian, and this lack of results motivated our investigation into such classes of elliptic problems.

\section{Preliminaries results}

In this section, we will discuss some fundamental tools in the theory of viscosity solutions to the inhomogeneous $\infty$-Laplacian. 

\begin{definition}[{\bf Viscosity solution}] 
We say that a function \( u \in C^0(\overline{\Omega}) \) is a sub-solution (resp. super-solution) of the PDE 
$$
-\Delta_{\infty} u(x) = f(|x|, u) \quad \text{in} \quad \Omega
$$ 
if, for every \( \varphi \in C^2(\overline{\Omega}) \) such that \( u - \varphi \) has a local maximum (resp. minimum) at some \( x_0 \in \Omega \), then
\[
-\Delta_{\infty} \varphi(x_0) \leq f(|x_0|, u(x_0)) \quad (\text{resp.} \,\,\,\geq f(|x_0|, u(x_0))).
\]
Finally, a function \( u \in C^0(\overline{\Omega}) \) is a viscosity solution of \( -\Delta_{\infty} u = f(|x|, u) \) if it is both a sub-solution and a super-solution.
\end{definition}

The first useful result is the maximum-minimum principle.

\begin{lemma}[{\cite[Lemma 2.1]{BhatMoh11}}]\label{MP} 
Let \( u \in C^0(\overline{\Omega}) \) satisfy \( -\Delta_{\infty} u \leq 0 \) (resp. \( -\Delta_{\infty} u \geq 0 \)) in the viscosity sense. Then,
\[
\sup_{\overline{\Omega}} u = \sup_{\partial \Omega} u \quad \left(\text{resp.} \quad \inf_{\overline{\Omega}} u = \inf_{\partial \Omega} u \right).
\]
Moreover, unless \( u \) is constant, the supremum (resp. infimum) occurs only on the boundary \( \partial \Omega \).
\end{lemma}

We must emphasize that the maximum principle stated above can be derived from Harnack’s inequality (see \textit{e.g.} \cite{ACJ04}).

To study the Dirichlet problem 
\begin{equation}\label{DirichetProbzero}
\left\{
\begin{array}{rclcc}
-\Delta_{\infty} u(x) & = & f(x, u) & \text{in} & \Omega \\
u(x) & = & g(x) & \text{on} & \partial \Omega 
\end{array}
\right.
\end{equation}
we will consider a continuous, one-sign function \( f: \Omega \times \mathbb{R} \to \mathbb{R} \), and we will assume one of the following conditions on the forcing term:
\begin{equation}\label{CC_forcing_term}
\sup_{x \in \overline{\Omega}} f(x, t) < \infty \quad \text{for each } t \in \mathbb{R}, \quad \left(
\inf_{x \in \overline{\Omega}} f(x, t) > -\infty \quad \text{for each } t \in \mathbb{R}\right).
\end{equation}

Finally, we will always assume that the boundary data \( g \), in problem \eqref{DirichetProbzero}, belongs to \( C^0(\partial \Omega) \).

Next, we will introduce the following classes of functions:
\[
\mathcal{N}_{+} := \left\{ w_1 \in C^0(\overline{\Omega}) \mid -\Delta_{\infty} w_1(x) \leq f(x, w_1) \text{ in } \Omega, \text{ and } w_1 \leq g \text{ on } \partial \Omega \right\},
\]
and
\[
\mathcal{N}_{-} := \left\{ w_2 \in C^0(\overline{\Omega}) \mid -\Delta_{\infty} w_2(x) \geq f(x, w_2) \text{ in } \Omega, \text{ and } w_2 \geq g \text{ on } \partial \Omega \right\}.
\]

We would like to emphasize that, according to \cite[Lemma 3.1]{BhatMoh11}, both \(\mathcal{N}_{\pm}\) are non-empty, provided that \(f\) satisfies the assumption in \eqref{CC_forcing_term}.

Now, we define the functions
\begin{equation}\label{Perron_sol}
u(x) := \sup_{w_1 \in \mathcal{N}_{+}} w_1(x) \quad \text{and} \quad v(x) := \inf_{w_2 \in \mathcal{N}_{-}} w_2(x) \quad (x \in \Omega).
\end{equation}

We note that the following result establishes the Lipschitz continuity of the functions defined in \eqref{Perron_sol}.

\begin{lemma}[{\cite[Lemma 3.4]{BhatMoh11}}]\label{Lipschitz continuous}
If \(f\) is non-negative, then the function \(u\) defined in \eqref{Perron_sol} is locally Lipschitz continuous on \(\Omega\). Similarly, if \(f\) is non-positive, then the function \(v\) defined in \eqref{Perron_sol} is locally Lipschitz continuous on \(\Omega\).
\end{lemma}

Now, we are in a position to present an existing result.

\begin{theorem}[{\cite[Theorem 3.5]{BhatMoh11}}]\label{Existence_Solution}
Suppose that \(f: \Omega \times  \mathbb{R} \to \mathbb{R}\) is continuous and non-decreasing in the second variable.
\begin{enumerate}
    \item If \(f\) is non-negative, then \(u\) given in \eqref{Perron_sol} is a solution of the Dirichlet problem \eqref{DirichetProbzero}.
    \item If \(f\) is non-positive, then \(v\) given in \eqref{Perron_sol} is a solution of the Dirichlet problem \eqref{DirichetProbzero}.
\end{enumerate}    
\end{theorem}

The following lemma presents a Comparison Principle result for the Dirichlet problem \eqref{DirichetProbzero}.

\begin{lemma}[{\cite[Lemma 4.1]{BhatMoh11}}]\label{LemmaCP}
    Suppose \(f_i: \Omega \times \mathbb{R} \to \mathbb{R}\) for \(i = 1, 2\) are continuous. Let \(u, v \in C(\overline{\Omega})\) satisfy 
    $$
    -\Delta_{\infty} u(x) = f_1(x, u) \quad \text{and} \quad -\Delta_{\infty} v(x) = f_2(x, v).
    $$
    Furthermore, assume that either \(f_1(x, t)\) or \(f_2(x, t)\) is non-decreasing in \(t\), and that \(f_1(x, t) > f_2(x, t)\) for all \((x, t) \in \Omega \times \mathbb{R}\). If \(u \leq v\) on \(\partial \Omega\), then \(u \leq v\) in \(\Omega\).
\end{lemma}

In the sequel, Lemma \ref{LemmaCP} allows us to obtain the following uniqueness result.

\begin{corollary}[{\cite[Corollary 4.6]{BhatMoh11}}]\label{Uniqueness}
 Let \(f: \Omega \times \mathbb{R} \to \mathbb{R}\) be continuous and either positive or negative. If \(f\) is also non-decreasing in the second variable, then the Dirichlet problem \eqref{DirichetProbzero} has at most one solution.
\end{corollary}

Next, we present a type of stability result.

\begin{lemma}[{\cite[Lemma 5.1]{BhatMoh11}}] \label{stability}  
Let \(\{\zeta_k\}_{k=1}^{\infty}\) be a sequence of non-negative functions in \(C^0(\overline{\Omega})\) such that \(\zeta_k \to \zeta\) locally uniformly in \(\Omega\) for some \(\zeta \in C^0(\overline{\Omega})\). Suppose that for each positive integer \(k\), \(u_k \in C(\overline{\Omega})\) is a solution of the Dirichlet problem
\begin{equation*}
\left\{
\begin{array}{rclccc}
-\Delta_{\infty} u_k(x) &=& \zeta_k(x) & \text{in} & \Omega, \\
u_k(x) &=& g(x) & \text{on} & \partial \Omega,
\end{array}
\right.
\end{equation*}
with the property that \(u_0 \leq u_k \leq u_{\infty}\) in \(\Omega\) for some functions \(u_0\) and \(u_{\infty}\) in \(C^0(\overline{\Omega})\), where \(u_0|_{\partial \Omega} = u_{\infty}|_{\partial \Omega} = g\). Then, the sequence \(\{u_k\}_{k \in \mathbb{N}}\) has a subsequence that converges locally uniformly in \(\Omega\) to a solution \(u \in C^0(\overline{\Omega})\) of the Dirichlet problem
\begin{equation*}
\left\{
\begin{array}{rclccc}
-\Delta_{\infty} u(x) &=& \zeta(x) & \text{in} & \Omega, \\
u(x) &=& g(x) & \text{on} & \partial \Omega.
\end{array}
\right.
\end{equation*}

\end{lemma}

\begin{remark}\label{Remark2.1}
We must stress that if the sequence \(\{u_k\}_{k \in \mathbb{N}}\) in the  Lemma \ref{stability} is monotonic, then the full sequence converges locally uniformly.
\end{remark}

Next, we will assume that \(f: \Omega \times \mathbb{R} \to \mathbb{R}\) satisfies the following condition: for every compact interval 
\(\mathrm{I} \subset \mathbb{R}\),
\begin{equation}\label{EqCond_f}
\displaystyle \sup_{\Omega \times \mathrm{I}} |f(x, t)| < \infty.
\end{equation}

\begin{theorem}[{\bf \(L^{\infty}\) bounds - \cite[Theorem 5.3]{BhatMoh12}}]\label{A priori L_infty bounds}
Let \(\Omega \subset \mathbb{R}^n\) be a bounded domain, \(g \in C^0(\partial \Omega)\), and \(f \in C^0(\Omega \times \mathbb{R}, \mathbb{R})\) such that \eqref{EqCond_f} holds. Assume that

\begin{equation}\label{Growth_Cond_f}
\left\{
\begin{array}{cc}
\text{(i)} & \quad \displaystyle \liminf_{t \to \infty} \inf \frac{f(x, t)}{t^3} := \mathfrak{L}_{-}, \\
\text{(ii)} & \quad \displaystyle \liminf_{t \to -\infty} \sup \frac{f(x, t)}{t^3} := \mathfrak{L}_{+}
\end{array}
\right.
\end{equation}
for some \(\mathfrak{L}_{\pm} \in [0, \infty]\). Then, there exists a constant \(\mathrm{C} > 0\), depending on \(f\), \(g\), and \(\text{diam}(\Omega)\), such that 
\[
\| u \|_{L^\infty(\Omega)} \leq \mathrm{C}
\]
for any solution \(u \in C^0(\Omega)\) of \eqref{DirichetProbzero}.
\end{theorem}

The next result provides a version of the Harnack inequality for the inhomogeneous \(\infty\)-Laplacian.

\begin{theorem}[{\bf Harnack inequality - \cite[Theorem 7.1]{BhatMoh12}}]\label{Harnack inequality}
Suppose \(h \in C^0(\Omega) \cap L^{\infty}(\Omega)\), and let \(u \in C^0(\Omega)\) be a non-negative viscosity super-solution of 
\[
\Delta_{\infty} u(x) = h(x) \quad \text{in} \quad \Omega.
\]
Then, for any \(z \in \Omega\) such that \(B_{2r}(z) \subseteq \Omega\), we have
\[
\sup_{\mathrm{B}} u(x) \leq 9 \inf_{\mathrm{B}} u(x) + 12 \sigma \left( r^4 \sup_{\Omega} h_{+}(x) \right)^{1/3},
\]
where \(\mathrm{B} := B_{2r/3}(z)\) and \(\sigma = \frac{3^{\frac{3}{4}}}{4}\).
\end{theorem}

The next theorem provides a straightforward version of local Lipschitz regularity for the inhomogeneous problem.

\begin{theorem}[{\bf Local Lipschitz regularity - \cite[Corollary 2]{Lind14}}]\label{Reg_Lipsc} 
Let \( u \in C^0(B_1) \) be a solution of 
\[
-\Delta_{\infty} u = f \in L^{\infty}(B_1) \cap C^0(B_1).
\]
Then, \( u \) is locally Lipschitz, and in particular,
\[
\| u \|_{C^{0,1}(B_{1/2})} \leq \mathrm{C} \cdot \left( \| u \|_{L^{\infty}(B_1)} + \| f \|_{L^{\infty}(B_1)}^{\frac{1}{3}} \right).
\]
\end{theorem}

\section{Proof of Theorem \ref{Hessian_continuity}}

In this section, we will present the proof of Theorem \ref{Hessian_continuity}.

\begin{proof}[{\bf Proof of Theorem \ref{Hessian_continuity}}]  For simplicity, and without loss of generality, we assume \( x_0 = 0 \). By combining discrete iterative techniques with continuous reasoning, see \cite{CKS00}, it is well established that proving estimate \eqref{Higher Reg} is equivalent to verifying the existence of a universal constant \( \mathrm{C} > 0 \), such that for all \( j \in \mathbb{N} \), the following holds:
\begin{equation}\label{higher1}
\mathfrak{s}_{j+1} \leq \max\left\{\mathrm{C} 2^{-\hat{\beta}(j+1)}, 2^{-\hat{\beta}} \mathfrak{s}_j\right\},
\end{equation}
where 
\begin{equation}\label{beta_hat}
\mathfrak{s}_j \coloneqq \sup_{B_{2^{-j}}} u \quad \text{and} \quad \hat{\beta} \coloneqq \frac{4 + \alpha}{3 - m}.
\end{equation}
Suppose, for the sake of contradiction, that \eqref{higher1} fails to hold, \textit{i.e.,} for each \( k \in \mathbb{N} \), there exists \( j_k \in \mathbb{N} \) such that
\begin{equation}\label{higher2}
\mathfrak{s}_{j_k + 1} > \max\left\{k 2^{-\hat{\beta}(j_k + 1)}, 2^{-\hat{\beta}} \mathfrak{s}_{j_k}\right\}.
\end{equation}

Now, for each \( k \in \mathbb{N} \), we define the rescaled function \( v_k: B_1 \rightarrow \mathbb{R} \) as
$$
v_k(x) \coloneqq \frac{u(2^{-j_k} x)}{\mathfrak{s}_{j_k + 1}}. 
$$
Note that
\begin{eqnarray}
0 \leq v_k(x) \leq 2^{\hat{\beta}}; \label{vk1}\\
v_k(0) = 0; \label{vk2}\\ 
\sup_{\overline{B}_{1/2}} v_k(x) = 1. \label{vk3}
\end{eqnarray}
Moreover, observe that
\begin{align}
\Delta_\infty v_k &= \langle D^2 v_k(x) D v_k(x), D v_k(x) \rangle \nonumber\\[0.2cm]
&= \left\langle \frac{2^{-2j_k}}{\mathfrak{s}_{j_k + 1}} D^2 u\left(2^{-j_k} x\right) \left(\frac{2^{-j_k}}{\mathfrak{s}_{j_k + 1}} D u\left(2^{-j_k} x\right)\right), \left(\frac{2^{-j_k}}{\mathfrak{s}_{j_k + 1}} D u\left(2^{-j_k} x\right)\right) \right\rangle \nonumber\\[0.2cm]
&= \frac{2^{-4j_k}}{\mathfrak{s}_{j_k + 1}^3} \Delta_\infty u\left(2^{-j_k} x\right) \nonumber\\[0.2cm]
&\coloneqq f_k. \label{v_k-solution}
\end{align}
Now, using \eqref{higher2}, we can estimate
$$
|f_k| \leq \frac{c_n \|f_0\|_{L^\infty(B_1)}}{2^{(4 + \alpha)(1 + 2j_k)} k^{3 - m}}.
$$

Next, by invoking the Harnack inequality (see Theorem \ref{Harnack inequality}), along with \eqref{vk2}, \eqref{vk3}, and \eqref{v_k-solution}, we obtain that
\begin{eqnarray*}
1 & = & \sup_{B_{1/2}} v_k(x)\\
& \leq & 9 \inf_{B_{1/2}} v_k(x) + 12 \sigma \left(\frac{1}{2}\right)^{\frac{4}{3}} \left( \sup_{B_{1/2}} f_k(x) \right)^{1/3}\\
& \leq & 12 \sigma \left(\frac{1}{2}\right)^{\frac{4}{3}} \left(\frac{c_n \|f_0\|_{L^\infty(B_1)}}{2^{(4 + \alpha)(1 + 2j_k)} k^{3 - m}}\right)^{\frac{1}{3}}\\
& \rightarrow & 0 \quad \text{as} \quad k \rightarrow \infty, 
\end{eqnarray*}
which yields a contradiction.

Finally, given \( r \in (0, 1) \), let \( j \in \mathbb{N} \) be such that
$$
2^{-(j + 1)} \leq r \leq 2^{-j}.
$$
Then, by using \eqref{higher1}, we establish the existence of a universal constant \( \mathrm{C} > 0 \) such that
$$
\sup_{B_r} u \leq \sup_{B_{2^{-j}}} u \leq \frac{\mathrm{C}}{2^{\hat{\beta}}} r^{\hat{\beta}}.
$$
Thus, the theorem is proven.
\end{proof}

As a consequence of the above theorem, we can obtain a finer control of the gradient near a free boundary point \( x_0 \) where the source function \( x \mapsto f(|x|, u) \) vanishes as \( \mathrm{dist}(x_0, \partial \{u > 0\})^{\alpha} \).

\begin{proof}[{\bf Proof of Corollary \ref{gradiente control}}]
    Let \( x_0 \in \{ u > 0\} \cap B_{1/2} \) be a point such that \( f(x, t) \simeq |x - \zeta(x_0)|^{\alpha} t_+^m \), where \( \zeta = \zeta(x_0) \in \partial \{ u > 0\} \) satisfies
    $$
        |x_0 - \zeta(x_0)| = d = \text{dist}(x_0, \partial \{ u > 0\}).
    $$
    From Theorem \ref{Hessian_continuity}, we know that 
    \begin{equation}\label{apply theo H}
        \sup\limits_{B_d(x_0)} u \le \sup\limits_{B_{2d}(\zeta)} u \le \mathrm{C}_1(\text{universal}) d^{\frac{4 + \alpha}{3 - m}}.
    \end{equation}
    Next, we set \( \beta := \frac{4 + \alpha}{3 - m} \) and consider the auxiliary function \( v: B_1 \to \mathbb{R} \) defined by 
    $$
        v(x) := \frac{u(x_0 + dx)}{d^{\beta}}. 
    $$
    It is easy to verify that \( v \) is a viscosity solution to 
    $$
        \Delta_{\infty} v = g(x, v), 
    $$
    where \( g(x, v) = d^{4 - 3\beta} f(|x_0 + dx|, d^{\beta} v) \). Moreover, we have 
    $$
    \sup\limits_{B_1} g(x, v) \le \mathfrak{L}_{\ast}(\text{universal}).
    $$
    Hence, from \eqref{apply theo H}, 
    $$
        \sup\limits_{B_1} v(x) \le \mathrm{C}_1(\text{universal}).
    $$
    Thus, from the gradient estimates for bounded solutions (see Theorem \ref{Reg_Lipsc}), we have 
    $$
\begin{array}{rcl}
    \frac{1}{d^{\beta - 1}} |D u(x_0)| & = & |D v(0)|  \\
    & \leq & \|v\|_{C^{0, 1}(B_{1/2})} \\
    & \leq & \mathrm{C} \left( \|v\|_{L^{\infty}(B_1)} + \|g\|^{\frac{1}{3}}_{L^{\infty}(B_1)} \right) \\
    & \leq & \mathrm{C}_2(\text{universal}).
\end{array}
$$
\end{proof}

\section{Non-degeneracy estimate: Proof of Theorem \ref{NãoDeg}}

In this section, we will prove Theorem \ref{NãoDeg}. 

To achieve this, for each $\varepsilon > 0$, we consider the following penalized problem:
\begin{equation}\label{Pe}
\tag{$\mathrm{P}_\varepsilon$} \Delta_\infty u_\varepsilon(x) = |x|^\alpha (u_\varepsilon)_+^m + \varepsilon \quad \text{in} \quad B_1, 
\end{equation}
where $u_\varepsilon = u$ on $\partial B_1.$

Observe that for each fixed $\varepsilon > 0$, the existence of such a solution $u_\varepsilon$ for the penalized problem \eqref{Pe} is an adaptation of Theorem \ref{Existence_Solution}. Additionally, we can express the right-hand side as $f_\varepsilon(x, t) = f(x, t) + \varepsilon$. Thus, we can ensure that $f_\varepsilon$ satisfies the condition \eqref{Growth_Cond_f}, and by Theorem \ref{A priori L_infty bounds}, the solutions $u_\varepsilon$ are bounded. Note that
$$
u \le u_{\varepsilon_1} \leq u_{\varepsilon_2} \cdots \leq u_{\varepsilon_j} \leq \cdots \leq \mathfrak{h}, 
$$
for every subsequence $\{\varepsilon_j\}_{j \in \mathbb{N}}$ with $\varepsilon_j \to 0^{+}$, where $\mathfrak{h}$ satisfies in the viscosity sense
$$
\Delta_{\infty} \mathfrak{h} = 0 \quad \text{in} \quad B_1 \quad \text{and} \quad \mathfrak{h}=u \quad \text{on} \quad \partial B_1.
$$
Finally, by Lemma \ref{Lipschitz continuous}, $u_\varepsilon$ is locally Lipschitz continuous. Consequently, the family $\{u_{\varepsilon}\}_{\varepsilon > 0}$ is precompact in the $C^{0,1}(B_{1})$ topology. Hence, by the Arzel\`{a}-Ascoli theorem, there exists a subsequence $\{u_{\varepsilon_{k_{j}}}\}_{k_{j}}$ that converges uniformly, i.e., there exists $u_{\infty} \in C^0(B_1)$ such that 

\begin{equation*}
    u_{\varepsilon_{k_{j}}} \to u_{\infty} \quad \text{as} \quad j \to \infty.
\end{equation*}

Next, our goal will be to prove the non-degeneracy estimate for solutions of \eqref{Pe} uniformly in $\varepsilon$.


\begin{proof}[{\bf Proof of Theorem \ref{NãoDeg}}]

Before addressing the proof of the desired result, we will show the corresponding estimate for viscosity solutions of \eqref{Pe} (uniformly in $\varepsilon$).

{\bf Step $1$:} By the radial analysis (see Example \ref{RadialEx}), we know that the function
$$
\Psi(x) = \left(\frac{(3-m)^4}{(4+\alpha)^3(1+\alpha+m)}\right)^{\frac{1}{3-m}} |x|^{\frac{4+\alpha}{3-m}}
$$
is a viscosity solution of 
$$
\Delta_\infty \Psi(x) = |x|^\alpha \Psi_+^m(x).
$$
We can assert the existence of a point $z_r \in \partial B_r(x_0)$ such that $u_\varepsilon(z_r) \geq \Psi(z_r)$. Indeed, if this were not the case, we could consider the functions $f_1(x, t) = |x|^\alpha t_+^m + \varepsilon$ and $f_2(x, t) = |x|^\alpha t_+^m$. By applying the Comparison Principle (see Lemma \ref{LemmaCP}), we would conclude that $\Psi \geq u_\varepsilon$ in the entire ball $B_r(x_0)$. However, this leads to the contradiction
$$
0 = \Psi(0) < u_\varepsilon(0).
$$
Therefore, we obtain the following estimate
\begin{eqnarray*}
\sup_{\partial B_r(x_0)} u_\varepsilon(x) & \geq & u_\varepsilon(z_r) \\ 
& \geq & \Psi(z_r) \\
& = & \left(\frac{(3-m)^4}{(4+\alpha)^3(1+\alpha+m)}\right)^{\frac{1}{3-m}} r^{\frac{4+\alpha}{3-m}},
\end{eqnarray*}
which concludes the desired estimate.

{\bf Step $2$:} Finally, by invoking Lemma \ref{stability} (resp. Remark \ref{Remark2.1}), and the uniqueness of the Dirichlet problem, we know that  
$$
 u_{\varepsilon_{k_{j}}} \to u_{\infty} = u \quad \text{as} \quad j \to \infty.
$$
Therefore, by uniform convergence, we conclude 
$$
\sup_{\partial B_r(x_0)} u(x) \leftarrow \sup_{\partial B_r(x_0)} u_{\varepsilon_j}(x)  \geq \left(\frac{(3-m)^4}{(4+\alpha)^3(1+\alpha+m)}\right)^{\frac{1}{3-m}} r^{\frac{4+\alpha}{3-m}},
$$
thereby finishing the proof.

\end{proof}

\section{Proof of Liouville-type Theorem \ref{Liouville I}}

As a consequence of Theorem \ref{Hessian_continuity}, we will prove Theorem \ref{Liouville I}.

\begin{proof}
For $k \in \mathbb{N}$, define
\begin{equation}\label{scal}
    u_{k}(x) = k^{-\frac{4+\alpha}{3-m}} u(kx), \quad x \in B_{1}.
\end{equation}
Note that $u_{k}(0) = 0$. Moreover, we have
$$
D u_k(x) = k^{\frac{{-1-\alpha-m}}{3-m}} Du(kx) \quad \text{and} \quad D^2u_k(x) = k^{\frac{2-\alpha-2m}{3-m}} D^2u(kx).
$$
Thus, we can compute
\begin{align*}
\Delta_\infty u_k(x) &= \left\langle D^2u_k(x) D u_k(x), D u_k(x) \right\rangle \\
&= \left\langle \left( k^{\frac{2-\alpha-2m}{3-m}} D^2u(kx) \right) \left( k^{\frac{{-1-\alpha-m}}{3-m}} D u(kx)\right), \left( k^{\frac{{-1-\alpha-m}}{3-m}} D u(kx)\right) \right\rangle \\
&= k^{\frac{-3\alpha-4m}{3-m}} \Delta_\infty u(kx) \\
&= k^{\frac{-3\alpha-4m}{3-m}} f_k(x, u_k), 
\end{align*}
where
$$
f_k(x, t) = k^{\frac{-3\alpha-4m}{3-m}} f(kx, k^{\frac{4+\alpha}{3-m}} t).
$$
It is important to note that $f_k$ satisfies \eqref{EqHomog-f}. Indeed, we have
\begin{align*}
|f_k(x, t)| & = k^{\frac{-3\alpha-4m}{3-m}} |f(kx, k^{\frac{4+\alpha}{3-m}} t)| \\
& \leq c_{n} k^{\frac{-3\alpha-4m}{3-m}} k^\alpha k^{({\frac{4+\alpha}{3-m}})m} |f_0(x)| \\
& = c_{n} k^{\frac{-3\alpha-4m}{3-m}} k^{\frac{\alpha(3-m)+(4+\alpha)m}{3-m}} |f_0(x)| \\
& = c_{n} |f_0(x)|.
\end{align*}

Now, let $x_{k} \in \overline{B}_{r}$ be such that

\begin{equation*}
    u_{k}(x_{k}) = \sup\limits_{\overline{B}_{r}} u_{k}.
\end{equation*}

Here, $r$ is small, and within $B_{r}$, we obtain

\begin{equation}\label{first}
    \|u_{k}\|_{L^{\infty}(B_{1})} \to 0 \quad \text{as} \quad k \to \infty.
\end{equation}

We will now analyze two cases. First, if $\lvert kx_{k}\rvert$ remains bounded as $k \to \infty$, then applying Theorem \ref{Hessian_continuity} to $u_{k}$ gives us
\begin{equation}\label{second}
    u_{k}(x_{k}) \leq \mathrm{C}_{k} \lvert x_{k}\rvert^{\frac{4+\alpha}{3-m}},
\end{equation}
where $\mathrm{C}_{k} > 0$ and $\mathrm{C}_{k} \to 0$ as $k \to \infty$. Recall equation \eqref{scal}, which leads to
\begin{equation*}
    u(kx_{k}) \leq \mathrm{C}_{k} \lvert k x_{k}\rvert^{\frac{4+\alpha}{3-m}}.
\end{equation*}

In other words, $u(kx_{k})$ remains bounded as $k$ approaches infinity, and therefore,
\begin{equation*}
    u_{k}(x_{k}) \to 0 \quad \text{as} \quad k \to \infty,
\end{equation*}
since $\mathrm{C}_{k}$ approaches zero. 

Now, if $\lvert kx_{k}\rvert \to \infty$ as $k \to \infty$, then, by hypothesis \ref{hypothesis}, we have

\begin{equation*}
    u_{k}(x_{k}) \leq \lvert kx_{k}\rvert^{\frac{4+\alpha}{3-m}} k^{\frac{-(4+\alpha)}{3-m}} \to 0 \quad \text{as} \quad k \to \infty.
\end{equation*}

Now, if there exists $z \in \mathbb{R}^{n}$ such that $u(z) > 0$, we can choose $\tau \in \mathbb{N}$ sufficiently large so that $z \in B_{\tau r}$. Using equations \ref{first} and \ref{second}, we obtain:

\begin{eqnarray*}
\frac{u(z)}{\lvert z \rvert^{\frac{4+\alpha}{3-m}}} & \leq & \sup\limits_{B_{\tau r}} \frac{u(x)}{\lvert x \rvert^{\frac{4+\alpha}{3-m}}} \\
& = & \sup\limits_{B_{r}} \frac{u_{k}(x)}{\lvert x \rvert^{\frac{4+\alpha}{3-m}}} \\
& \leq & \frac{u(z)}{2 \lvert z \rvert^{\frac{4+\alpha}{3-m}}}
\end{eqnarray*}

This leads to a contradiction.

\end{proof}

\section{Radial analysis and proof of Theorem \ref{Liouville II}}

In this final section, we take a brief pause to analyze the radial boundary value problem
\begin{equation}\label{Radial_Analysis}
    \left\{
\begin{array}{rclcc}
\Delta_\infty u(x)  & = & (|x-x_0|-r)^\alpha u_+^m  & \text{in} & B_{\mathrm{R}}(x_0),\\
    u(x) & = & \mathrm{c} & \text{on} & \partial B_{\mathrm{R}}(x_0), 
\end{array}
    \right.
\end{equation}
where $\mathrm{c}>0$ is a constant, $0<r<\mathrm{R}$, and $x_0\in \mathbb{R}^n$. 

Next, we consider the following ODE related to \eqref{Radial_Analysis}:
\begin{equation}\label{radial-1D}
\omega''(s)(\omega'(s))^2=s^\alpha \omega(s)_+^m\quad\text{in}\quad (0,\mathrm{T}),
\end{equation}
subject to the initial conditions $\omega(0)=0$ and $\omega(\mathrm{T})=c$. Solving \eqref{radial-1D}, we obtain the solution
$$
\omega(s)= \tau(m,\alpha) s^{\frac{4+\alpha}{3-m}},
$$
where
$$
\tau(m,\alpha) = \left(\frac{(3-m)^4}{(4+\alpha)^3(1+\alpha+m)}\right)^{\frac{1}{3-m}}\quad \text{and}\quad \mathrm{T}=\left(\frac{\mathrm{c}}{\tau(m,\alpha)}\right)^{\frac{3-m}{4+\alpha}}.
$$
Fixing $x_0\in \mathbb{R}^n$ and $0 < r < \mathrm{R}$, we assume the dead-core compatibility condition
\begin{equation}\label{dead-core compatibility}
\mathrm{R}>\mathrm{T}.
\end{equation}
Define the following radially symmetric function $u:B_{\mathrm{R}}(x_0)\backslash B_r(x_0)\rightarrow \mathbb{R}$ given by
$$
u(x)\coloneqq \omega(|x-x_0|-r),
$$
where $r=\mathrm{R}-\mathrm{T}$. It can be easily verified that $u$ point-wise solves the equation
$$
\Delta_\infty u(x) = (|x-x_0|-r)^\alpha u_+^m(x)\quad\text{in}\quad B_{\mathrm{R}}(x_0)\backslash B_r(x_0).
$$
The boundary conditions $u\equiv 0$ on $\partial B_r(x_0)$ and $u=\mathrm{c}$ on $\partial B_{\mathrm{R}}(x_0)$ are also satisfied. Moreover, for each $y\in \partial B_r(x_0)$, we have
$$
\lim_{x\rightarrow y} D u(x) = \omega'(0^+)\frac{y}{|y|}=0.
$$
Thus, by extending $u\equiv 0$ in $B_r(x_0)$, we obtain a function in $B_{\mathrm{R}}(x_0)$ satisfying 
$$
\Delta_\infty u(x) = (|x-x_0|-r)^\alpha u_+^m(x)\quad\text{in}\quad B_{\mathrm{R}}(x_0).
$$
We conclude that the function
$$
u(x)=\tau(m,\alpha)\left(|x-x_0|-\mathrm{R}+\left(\frac{\mathrm{c}}{\tau(m,\alpha)}\right)^{\frac{3-m}{4+\alpha}}\right)^{\frac{4+\alpha}{3-m}}
$$
is the solution to \eqref{Radial_Analysis}. Its plateau is precisely $B_r(x_0)$, where
$$
0<r=\mathrm{R}- \left(\frac{\mathrm{c}}{\tau(m,\alpha)}\right)^{\frac{3-m}{4+\alpha}}.
$$

Finally, we are in a position to address the Theorem \ref{Liouville II}.

\begin{proof}[{\bf Proof of Theorem \ref{Liouville II}}]
For each $\varepsilon>0$, let \( u_\varepsilon \) be the solution of the penalized problem \eqref{Pe} with $|x|^{\alpha}$ replaced by $(|x-x_0|-r)_+^{\alpha}$, and $u_\varepsilon(x)  =  \displaystyle\sup_{\partial B_r(x_0)} u(x)$ on  $\partial B_{\mathrm{R}}(x_0)$.

Now, fixing \( \mathrm{R}>r >0 \), we consider \( v: B_{\mathrm{R}}(x_0) \to \mathbb{R} \), the radial solution to the boundary value problem
$$
\left\{
\begin{array}{rclcc}
\Delta_\infty v  & = & (|x-x_0|-r)_+^\alpha v_+^m  & \text{in} & B_{\mathrm{R}}(x_0),\\
    v(x) & = & \displaystyle\sup_{\partial B_r(x_0)} u(x) & \text{on} & \partial B_{\mathrm{R}}(x_0).
\end{array}
\right.
$$
First, by using the assumption $\mathfrak{h}_{r, x_0}(|x|) \lesssim (|x-x_0|-r)_+^{\alpha}$, we can apply the Comparison Principle (see, Lemma \ref{LemmaCP}) to ensure that
\begin{equation}\label{Liouville II1}
u \leq u_\varepsilon \quad \text{in } B_{\mathrm{R}}(x_0).
\end{equation}

Since \( u_{\varepsilon_j} \to u_0 \) uniformly in compact sets, we use the Stability and the Uniqueness of solutions to the Dirichlet problem ( Lemma \ref{stability} (resp. Remark \ref{Remark2.1}) and Corollary \ref{Uniqueness}) to conclude that \( v = u_0 \). Thus, from \eqref{Liouville II1}, we have
\begin{equation}\label{Liouville II2}
u \leq v \quad \text{in } B_{\mathrm{R}}(x_0).
\end{equation}
By hypothesis \eqref{Liouville II-H}, taking \( \mathrm{R} \gg 1 \), it follows that
\begin{equation}\label{Liouville II3}
    \sup_{\partial B_{\mathrm{R}}(x_0)}\frac{u(x)}{\mathrm{R}^{\frac{4+\alpha}{3-m}}} \leq \tau(m,\alpha) \theta,
\end{equation}
for some \( \theta \in (0, 1) \). For \( \mathrm{R} \gg 1 \), the solution \( v \) is given by
\begin{equation}\label{Liouville II4}
v(x) = \tau(m,\alpha)\left(|x-x_0|-\mathrm{R}+\left(\frac{\displaystyle\sup_{\partial B_{\mathrm{R}}(x_0)}u}{\tau(m,\alpha)}\right)^{\frac{3-m}{4+\alpha}}\right)_+^{\frac{4+\alpha}{3-m}}.
\end{equation}
Finally, combining \eqref{Liouville II3} and \eqref{Liouville II4}, we obtain
$$
u(x) \leq \tau(m,\alpha)\left(|x-x_0|-\mathrm{R}\left(1-\theta^{\frac{3-m}{4+\alpha}}\right)\right)_{+}^{\frac{4+\alpha}{3-m}}.
$$
Letting \( \mathrm{R} \to \infty \), we conclude the proof of the theorem.
\end{proof}

\subsection*{Final comments and generalizations}

A natural extension of our results arises when considering equations of the form 
\begin{equation}\label{extension}
    \Delta_{\infty} u = f(x, u, Du) \quad \text{in} \quad \Omega,
\end{equation}
where the right-hand side decays as $|f(x, t, \xi)| \le \mathrm{c}_n |x|^{\alpha} |t|^{m} |\xi|^\gamma$ for a dimensional constant $\mathrm{c}_n>0$. Allowing independent decay of \(f\) concerning \(x\), \(u\), and \(Du\) results in a gain of regularity in all possible scenarios, as will be discussed below.

A careful analysis in the proof of Theorem \ref{Hessian_continuity} yields that at the vanishing critical points, i.e., 
$$
   x_0 \in  \mathcal{Z}^{\prime}_0(u, \Omega) = \{ x \in \Omega : u(x) = | Du(x)| = 0\} 
$$
where 
\(
f(|x|, t, \xi) \simeq |x - x_0|^{\alpha} |t|^m |\xi|^{\kappa}
\), the solutions to \eqref{extension} belong to the class $C_{\text{loc}}^{\frac{4 + \alpha - \kappa}{3 - m - \kappa}}$ along  the set$\mathcal{Z}^{\prime}_0(u, \Omega)$. Moreover, such results will appear in the forthcoming work \cite{daSdosSS24}.

The relevance of studying equations like \eqref{extension} and their regularity along critical points lies in the fact that this model is closely related to a class of anisotropic elliptic operators of the form 
\begin{equation}
    \Delta_{F, \infty}^N u := \langle D^2 u DF(Du), DF(Du) \rangle = \sum\limits_{i,j=1}^n \partial_{i} F(Du) \partial_{ij} u \partial_j F(Du)
\end{equation}
 where $F(\xi) = |\xi|^{\kappa}$. Note that the choice \(\kappa = 1\) leads to the normalized $\infty$-Laplacian
$$
    \Delta_{\infty}^N u = |Du|^{-2} \langle D^2 u Du, Du \rangle,
$$
and the choice \(\kappa = 2 - \frac{\gamma}{2}\) leads to the singular/degenerate operator
\begin{equation}
    \Delta_{F, \infty}^N u  = |Du|^{-\gamma} \Delta_{\infty} u = \Delta_{\infty}^{\gamma} u, 
\end{equation}
where  $\Delta_{\infty}^{\gamma} u$ is the $(3-\gamma)$-homogeneous operator related to the $\infty$-Laplacian (see, \cite{BisVo23} for related topics).

Thus, the model equation \eqref{extension} can itself be interpreted as a model for singular/degenerate equations of the form (see, \cite{LY12})
\begin{equation}\label{Aniso-gamma}
    \Delta_{\infty}^{\gamma} u = f(|x|, u). 
\end{equation}
The singular nature of the equation requires a ``very weak'' sense of solutions, where test functions with zero gradients are disregarded. The results in this paper convey that the set of critical points carries a richer regularity theory. Indeed, when \(\alpha\) and \(m \) approach zero, we can show that, along the critical set \(\mathcal{Z}^{\prime}_0(u, \Omega) \), solutions are asymptotically of class $C^{1,\frac{1}{3-\gamma}}$. In particular, this implies that solutions to the normalized $\infty$-Laplacian, which corresponds to the choice $\gamma = 2$ in \eqref{Aniso-gamma}, with a bounded right-hand side, are asymptotically of class $\displaystyle C^{1,1}$ along $\mathcal{Z}^{\prime}_0(u, \Omega) $.

Finally, we may consider the model problem with gradient dependence and drift terms as follows (see, \cite{BisVo23})
\[
\Delta^{\gamma}_{\infty} u + q(|x|) \cdot D u |D u|^{\kappa} = f(|x|,u) \lesssim |x|^{\alpha}u_{+}^m  \quad \text{in} \quad  B_1 \subset \mathbb{R}^n.
\]
where $0 \leq \kappa <2-\gamma$ and $q(|x|) \lesssim |x|^{\sigma}$. In this setting, we may prove that viscosity solutions satisfy the estimate
$$
\displaystyle \sup_{B_r(x_0)} u(x) \leq \mathrm{C}(\textit{universal})r^{\min\left\{\frac{3-\kappa-\gamma-\sigma}{2-\gamma-\kappa}, \frac{4-\gamma+\alpha}{3-m-\gamma}\right\}}
$$
along free boundary points $x_0 \in \mathcal{Z}^{\prime}_0(u, B_1)$. Moreover, observe that
$$
\frac{3-\kappa-\gamma-\sigma}{2-\gamma-\kappa}= 1+\frac{1-\sigma}{2-\gamma-\kappa}>1 \quad \Leftrightarrow \quad  \kappa <2-\gamma \quad \text{and} \quad 0 \le \sigma<1,
$$
and
$$
 \frac{4-\gamma+\alpha}{3-m-\gamma}>1 \quad \Leftrightarrow \quad  1 + \alpha + m > 0 \quad \text{and} \quad 4+\alpha> \gamma,
$$
thereby providing compatibility conditions between the exponents.



\begin{thebibliography}{99}



\bibitem{ALT16}  D.J. Ara\'{u}jo, R. Leit\~{a}o  and E. Teixeira, 
\textit{Infinity Laplacian equation with strong absorptions}. J. Funct. Anal. 270 (2016), no.6, 2249–2267.

\bibitem{ArmSmart10} S.N. Armstrong and C.K. Smart, 
\textit{An easy proof of Jensen's theorem on the uniqueness of infinity harmonic functions}.
Calc. Var. Partial Differential Equations 37 (2010), no.3-4, 381–384.

\bibitem{Aronsson67}  G. Aronsson, 
\textit{ Extension of functions satisfying Lipschitz conditions}, Ark. Mat. 6 (1967) 551-561.

\bibitem{Aronsson68}  G. Aronsson, 
\textit{On the partial differential equation $u^2_xu_{xx} + 2u_xu_yu_{xy} + u^2_yu_{yy} = 0$}, Arkiv f\"{o}r Matematik 7 (1968), pp. 395–435.

\bibitem{Aronsson84}  G. Aronsson,
\textit{On certain singular solutions of the partial differential
equation $u^2_xu_{xx} + 2u_xu_yu_{xy} + u^2
_yu_{yy} = 0$}, Manuscripta Mathematica 47 (1984), pp. 133-151.


\bibitem{ACJ04} G. Aronsson,  M.G. Crandall  and  P. Juutinen, 
\textit{A tour of the theory of absolutely minimizing functions}.  Bull. Amer. Math. Soc. (N.S.) 41 (2004), no.4, 439–505.

\bibitem{BEJ08} E.N. Barron, L.C. Evans and R. Jensen,  
\textit{The infinity Laplacian, Aronsson’s equation and
their generalizations}. 
Trans. Amer. Math. Soc., 360 (1) :77-101, 2008.

\bibitem{Batt77} J. Batt,  
\textit{Global symmetric solutions of the initial value problem of stellar dynamics}. J. Differential Equations, 25 (3):342-364, 1977.

\bibitem{BFH86} J. Batt, W. Faltenbacher and  E. Horst,
\textit{Stationary spherically symmetric models in stellar
dynamics}. 
Arch. Rational Mech. Anal., 93 (2) :159-183, 1986.

\bibitem{BhatMoh11} T. Bhattacharya and  A. Mohammed,  
\textit{On solutions to Dirichlet problems involving the infinity-Laplacian}. Adv. Calc. Var. 4 (2011), no.4, 445–487.

\bibitem{BhatMoh12} T. Bhattacharya and  A. Mohammed,  
\textit{Inhomogeneous Dirichlet problems involving the infinity-Laplacian}. Adv. Differential Equations 17 (2012), no.3-4, 225–266.

\bibitem{BisVo23} A. Biswas and H.-H. Vo, 
\textit{Liouville theorems for infinity Laplacian with gradient and KPP type equation}. Ann. Sc. Norm. Super. Pisa Cl. Sci. (5) 24 (2023), no.3, 1223–1256.

\bibitem{BW06} J. Byeon and Z.Q. Wang, 
\textit{On the H\'{e}non equation: asymptotic profile of ground states}. 
I, Ann. Inst. H. Poincar\'{e} Anal. Non Lin\'{e}aire 23 (2006), no. 6, 803–828.

\bibitem{CKS00} L.A. Caffarelli, L. Karp and  H. Shahgholian,  
\textit{Regularity of a free boundary with application to the Pompeiu problem}. Ann. of Math. (2) 151 (2000), no.1, 269–292.

\bibitem{CKN84} L. Caffarelli, R. Kohn and L. Nirenberg,
\textit{First order interpolation inequalities with weights}.
Compositio Math. 53 (1984), no.3, 259–275.


\bibitem{CPY09} D. Cao, S. Peng and S. Yan,  
\textit{Asymptotic behaviour of ground state solutions for the H\'{e}non equation}, 
IMA J. Appl. Math. 74 (2009), no. 3, 468–480.

\bibitem{CotLa19} Cotsiolis, A. and Labropoulos, N.
\text{On the Hardy-Sobolev inequalities}. Differential and integral inequalities, 265–287.
Springer Optim. Appl., 151
Springer, Cham, [2019], ©2019
ISBN:978-3-030-27406-1 ISBN:978-3-030-27407-8.


\bibitem{daSdosPRS} J.V. da Silva, D.S dos Prazeres, G.C. Ricarte and  G.S. S\'{a}, \textit{Sharp and improved regularity estimates for weighted quasilinear elliptic equations of $p-$Laplacian type and applications}. Submitted, 2024.

\bibitem{daSdosSS24}  J.V. da Silva, M.S. dos Santos and M. Soares, \textit{Higher regularity of solutions to inhomogeneous infinity-Laplacian type equations}. In preparation, 2024.


\bibitem{daSRosSal19} J.V. da Silva, J.D. Rossi   and A.M. Salort, \textit{Regularity properties for $p$-dead core problems and their asymptotic limit as $p \to \infty$}.
J. Lond. Math. Soc. (2) 99 (2019), no.1, 69-96.

\bibitem{daSSal18}  J.V. da Silva and A.M. Salort, 
\textit{Sharp regularity estimates for quasi-linear elliptic dead core problems and applications}.
Calc. Var. Partial Differential Equations 57 (2018), no.3, Paper No. 83, 24 pp.


\bibitem{Diaz85} J.I. D\'{i}az, 
\textit{Nonlinear partial differential equations and free boundaries}. Vol. I.
Elliptic equations. Res. Notes in Math., 106 Pitman (Advanced Publishing Program), Boston, MA, 1985. vii+323 pp.
ISBN:0-273-08572-7.


\bibitem{DT20} N.M.L. Diehl and R. Teymurazyan, 
\textit{Reaction-diffusion equations for the infinity Laplacian}. Nonlinear Anal. 199 (2020), 111956, 12 pp.

\bibitem{dosSP16} E.M. dos Santos and F. Pacella, 
\textit{H\'{e}non-type equations and concentration on spheres}.
Indiana Univ. Math. J. 65 (2016), no.1, 273–306.


\bibitem{ES08}  L.C. Evans and  O. Savin, 
\textit{$C^{1,\alpha}$ regularity for infinity harmonic functions in two dimensions}. Calc. Var. Partial Differential Equations 32 (2008), no.3, 325–347.

\bibitem{ESmart11} L.C. Evans and C.K. Smart,  
\textit{Everywhere differentiability of infinity harmonic functions}. Calc. Var. Partial Differential Equations 42 (2011), no.1-2, 289–299.

\bibitem{Henon73} M. H\'{e}non,  
\textit{Numerical experiments on the stability of spherical stellar systems}, Astronomy and astrophysics 24 (1973), 229-238. 

\bibitem{Jensen93} R. Jensen,  
\textit{Uniqueness of Lipschitz extensions: minimizing the sup norm of the gradient}. Arch. Rational Mech. Anal. 123 (1993), no.1, 51–74.


\bibitem{LS95} Y. Li and J. Santanilla, 
\textit{Existence and nonexistence of positive singular solutions for semilinear elliptic problems with applications in astrophysics}.
Differential Integral Equations 8 (1995), no.6, 1369-1383.


\bibitem{Lind14} E. Lindgren,  
\textit{On the regularity of solutions of the inhomogeneous infinity Laplace equation}. 
Proc. Amer. Math. Soc. 142 (2014), no.1, 277–288.

\bibitem{LY12} F. Liu and X.-P. Yang, 
\textit{Solutions to an inhomogeneous equation involving infinity Laplacian}. Nonlinear Anal. 75 (2012), no.14, 5693–5701.

\bibitem{LuWang08} G. Lu and P. Wang, 
\textit{Inhomogeneous infinity Laplace equation}. Adv. Math. 217 (2008), no.4, 1838–1868.

\bibitem{Ni82} W.M. Ni,  
\textit{A nonlinear Dirichlet problem on the unit ball and its applications}, 
Indiana Univ. Math. J. 31 (1982), no. 6, 801–807.

\bibitem{PSSW09} Y. Peres, O. Schramm, S. Sheffield and  D.B. Wilson,  
\textit{Tug-of-war and the infinity Laplacian}. J. Amer. Math. Soc., 22 (1) : 167-210, 2009.

\bibitem{RTU15} J.D. Rossi, E.V. Teixeira and  J.M. Urbano,
\textit{Optimal regularity at the free boundary for the infinity obstacle problem}. Interfaces Free Bound. 17 (2015), no.3, 381–398.

\bibitem{Savin05} O. Savin,  
\textit{$C^1$ regularity for infinity harmonic functions in two dimensions}. Arch. Ration. Mech. Anal. 176 (2005), no.3, 351–361.

\bibitem{SWS02} D. Smets, M. Willem and J. Su, 
\textit{Non-radial ground states for the H\'{e}non equation}, 
Commun. Contemp. Math. 4 (2002), no. 3, 467–480.


\bibitem{Tei22} E.V. Teixeira, 
\textit{On the critical point regularity for degenerate diffusive equations}.
Arch. Ration. Mech. Anal. 244 (2022), no.2, 293–316.

\bibitem{Yu06} Y. Yu, \textit{A remark on $C^2$ infinity-harmonic functions}, Electronic Journal of Differential Equations 2006 no 122 (2006), pp. 1-4.

\bibitem{Zaj87} L. Zaj\'{i}\u{c}ek,   \textit{Porosity and  $\sigma-$porosity}, Real Anal. Exchange 13 (1987/88) 314-350.

\end{thebibliography}
\end{document}